\documentclass{article}
\usepackage{algorithm}
\usepackage{algpseudocode}
\usepackage{amsmath}
\usepackage{amssymb}
\usepackage{amsthm}
\usepackage[margin=1.0in]{geometry}
\usepackage{graphicx}
\usepackage{mathtools}
\usepackage{optidef}
\usepackage{xcolor}
\usepackage{tikz}
\usepackage{hyperref}

\DeclareMathOperator{\Z}{Z}
\DeclareMathOperator{\pt}{pt}
\DeclareMathOperator{\PT}{PT}
\DeclareMathOperator{\pti}{PTI}
\DeclareMathOperator{\amf}{AMF}
\DeclareMathOperator{\ft}{ft}
\DeclareMathOperator{\M}{M}
\DeclareMathOperator{\mr}{mr}
\DeclareMathOperator{\cc}{cc}
\let\th\relax
\DeclareMathOperator{\th}{th}
\DeclareMathOperator{\diam}{diam}

\newcommand\abs[1]{\left|#1\right|}
\newcommand\feas[1]{\operatorname{Feas}\left(#1\right)}
\newcommand\opt[1]{\operatorname{Opt}\left(#1\right)}
\newcommand\im[2]{\operatorname{IM}\left(#1,#2\right)}
\newcommand\impt[2]{\operatorname{IM}_{\pt}\left(#1,#2\right)}
\newcommand\imth[2]{\operatorname{IM}_{\th}\left(#1,#2\right)}
\newcommand\tsm[2]{\operatorname{TSM}\left(#1,#2\right)}
\newcommand\tsmpt[2]{\operatorname{TSM}_{\pt}\left(#1,#2\right)}
\newcommand\tsmPT[2]{\operatorname{TSM}_{\PT}\left(#1,#2\right)}
\newcommand\tsmth[2]{\operatorname{TSM}_{\th}\left(#1,#2\right)}
\newcommand\tsmpti[3]{\operatorname{TSM}_{\pti}\left(#1,#2,#3\right)}
\newcommand\fc[1]{\operatorname{FC}\left(#1\right)}
\newcommand\fn[1]{\operatorname{FN}\left(#1\right)}
\newcommand\lfc[1]{\operatorname{FC}^{*}\left(#1\right)}
\newcommand\mf[2]{\operatorname{MF}\left(#1,#2\right)}
\newcommand\lmf[2]{\operatorname{MF}^{*}\left(#1,#2\right)}
\newcommand\cl[1]{\operatorname{cl}\left(#1\right)}
\newcommand\mff[2]{\operatorname{MFF}\left(#1,#2\right)}

\newtheorem{theorem}{Theorem}[section]
\newtheorem{lemma}[theorem]{Lemma}
\newtheorem{corollary}[theorem]{Corollary}

\newtheorem{conjecture}[theorem]{Conjecture}
\theoremstyle{definition}
\newtheorem{example}[theorem]{Example}

\begin{document}
\title{IP models for minimum zero forcing sets, forts, and related graph parameters}
\author{Thomas R. Cameron\footnote{Department of Mathematics, Penn State Behrend {\tt trc5475@psu.edu}} \and Jonad Pulaj\footnote{Department of Mathematics and Computer Science, Davidson College {\tt jopulaj@davidson.edu}}}
\maketitle

\begin{abstract}
Zero forcing is a binary coloring game on a graph where a set of filled vertices can force non-filled vertices to become filled following a color change rule.
In 2008, the zero forcing number of a graph was shown to be an upper bound on its maximum nullity. 
In addition, the combinatorial optimization problem for the zero forcing number was shown to be NP-hard.
Since then, the study of zero forcing and its related parameters has received considerable attention. 
In 2018, the forts of a graph were defined as non-empty subsets of vertices where no vertex outside the set has exactly one neighbor in the set.
Forts have been used to model zero forcing as an integer program and provide lower bounds on the zero forcing number. 
To date, three integer programming models have been developed for the zero forcing number of a graph: the Infection Model, Time Step Model, and Fort Cover Model. 
In this article, we present variations of these models for computing the zero forcing number and related graph parameters, such as the minimum and maximum propagation times, throttling number, and fractional zero forcing number.
In addition, we present several new models for computing the realized propagation time interval, all minimal forts of a graph, and the fort number of a graph. 
We conclude with several numerical experiments that demonstrate the effectiveness of our models when applied to small and medium order graphs.
Moreover, we provide experimental evidence for several open conjectures regarding the propagation time interval, the number of minimal forts, the fort number, and the fractional zero forcing number of a graph. 
\end{abstract}
\section{Introduction}\label{sec:intro}
Zero forcing is a binary coloring game on a graph where a set of filled vertices can force non-filled vertices to become filled following a color change rule.
The zero forcing number is the minimum number of initially filled vertices necessary to force all vertices to become filled after repetitive application of a color change rule. 
In 2008, the zero forcing number of a graph was shown to be an upper bound on the maximum nullity of corresponding real symmetric matrices~\cite{AIM2008}.
Since then, the zero forcing number has been extensively studied and many results have been published, for example, on the known values for certain families of graphs~\cite{AIM2008}, extreme values~\cite{Ekstrand2013,Hogben2007}, and the effect of certain graph operations~\cite{Benson2018,Ekstrand2013,Row2012}.
In addition, zero forcing has been associated with a variety of applications, for example, the controllability of networked systems~\cite{Trefois2015}, quantum systems~\cite{Burgarth2007}, and logic circuits~\cite{Burgarth2015}.

Related graph parameters, such as the propagation time and throttling number have also received a considerable amount of attention in the literature, for example, on known values for certain families of graphs~\cite{Butler2013,Carlson2021,Chilakamarri2012,Hogben2012}, extreme values~\cite{Carlson2021}, and the effect of certain graph operations~\cite{Hogben2022}.
Moreover, applications of propagation time and throttling number include the controllability of quantum systems~\cite{Severini2008} and the phase measurement unit problem in electrical engineering~\cite{Brueni2005}.

The combinatorial optimization problem for the zero forcing number is NP-hard since the corresponding decision problem was shown to be NP-complete~\cite{Aazami2008}.
In~\cite{Butler2014}, a dynamic programming algorithm is developed for the computation of the zero forcing number; in~\cite{Agra2019,Brimkov2019zf,Brimkov2021} boolean satisfiability and integer programming models are developed for the computation of the zero forcing number. 
In~\cite{Brimkov2019th}, the authors show that the decision problem corresponding to the throttling number is NP-complete; hence, the combinatorial optimization problem for the throttling number is NP-hard.
Moreover, they provide a C implementation of a brute-force algorithm for computing the throttling number.

In~\cite{Brimkov2019zf}, the authors present two distinct integer programming models for the zero forcing number.
The first is called the Infection Model, since it models the zero forcing game as a dynamic graph infection process.
The Infection Model stores all information about the zero forcing game such as the initial set of filled vertices, the time step at which each vertex is filled, and what forcings occurred. 
While the Infection Model only has a polynomial number of constraints, some of these constraints are of big-M form.
Big-M constraints can lead to a large gap between the optimal value of the integer program and its linear programming relaxation, which in turn may lead to poor performance. 

The second is called the Fort Cover Model since it models the zero forcing sets of a graph as covers for the forts of the graph; that is, a subset of vertices is a zero forcing set if and only if it intersects every fort of the graph~\cite[Theorem 8]{Brimkov2019zf}. 
The forts of a graph were originally introduced in~\cite{Fast2018} and are closely related to the failed zero forcing sets of a graph~\cite{Fetcie2015,Swanson2023}; in particular, a set of vertices is a minimal fort, with respect to set inclusion, if and only if its complement is a maximal failed zero forcing set~\cite[Theorem 10]{Becker2025}. 
While the Fort Cover Model contains no big-M constraints, there are graphs with an exponential number of minimal forts, which leads to an exponential number of constraints~\cite{Becker2025,Brimkov2021}.
Hence, the Fort Cover Model requires constraint generation.
In~\cite{Brimkov2019zf}, the most efficient constraint generation technique relied on finding a minimum fort in the complement of a stalled failed zero forcing set.
It is worth noting that computing a minimum fort is NP-hard since the decision problem associated with a maximum failed zero forcing set is NP-complete~\cite{Shitov2017}.

In~\cite{Cameron2023}, the authors define the optimal value of the Fort Cover Model linear programming relaxation as the fractional zero forcing number.
In addition, they define the fort number as the largest number of pairwise disjoint forts. 
Both the fort number and fractional zero forcing number are used to provide bounds on the zero forcing number of Cartesian products.
Moreover, they conjecture that both are lower bounds for the maximum nullity of a graph. 

In~\cite{Agra2019}, the authors develop the Time Step Model for the zero forcing number of a graph. 
This model is similar to the Infection Model in that it models the zero forcing game as a dynamic graph infection process and retains all information about the zero forcing game, such as the initial set of filled vertices, the time step at which each vertex is filled, and what forcings occurred. 
Moreover, the Time Step Model avoids big-M constraints and was shown to be stronger than the Infection Model; that is, every feasible solution of the linear relaxation of the Time Step Model is also feasible for the relaxation of the Infection Model, and there are feasible solutions to the linear relaxation of the Infection Model that are not feasible for the relaxation of the Time Step Model.

In this article, we present variations of the Infection Model, Time Step Model, and Fort Cover Model to compute the zero forcing number and related graph parameters, such as the minimum and maximum propagation times, throttling number, and fractional zero forcing number. 
In addition, we present several new models for computing the realized propagation time interval, all minimal forts of a graph, and the fort number of a graph.
We conclude with several numerical experiments that demonstrate the effectiveness of our models when working with small and medium order graphs. 
Moreover, we provide experimental evidence for several open conjectures regarding the propagation time interval, the number of minimal forts, the fort number, and the fractional zero forcing number of a graph~\cite{Becker2025,Brimkov2025,Cameron2023}.
\section{Preliminaries}\label{sec:prelim}
In this section, we introduce definitions and preliminary results that will be used throughout this article.
In addition, we state several open conjectures that will be tested in Section~\ref{sec:num_exp}. 
\subsection{Graph definitions}\label{subsec:prelim_graph_defs}
Throughout this article, we let $\mathbb{G}$ denote the set of all finite simple unweighted graphs.
For each $G\in\mathbb{G}$, we have $G=(V,E)$, where $V$ is the vertex set, $E$ is the edge set, and $\{u,v\}\in E$ if and only if $u\neq v$ and there is an edge between vertices $u$ and $v$.
If the context is not clear, we use $V(G)$ and $E(G)$ to specify the vertex set and edge set of $G$, respectively. 
The \emph{order} of $G$ is denoted by $n=\abs{V}$ and the \emph{size} of $G$ is denoted by $m=\abs{E}$.
When $m=0$, we refer to $G$ as an \emph{empty graph} of order $n$; when $m=\binom{n}{2}$, we refer to $G$ as a \emph{complete graph} of order $n$.

Given $u\in V$, we define the \emph{neighborhood} of $u$ as $N(u) = \left\{v\in V\colon\{u,v\}\in E\right\}$.
We refer to every $v\in N(u)$ as a \emph{neighbor} of $u$, and we say that $u$ and $v$ are \emph{adjacent}. 
If the context is not clear, we use $N_{G}(u)$ to denote the neighborhood of $u$ in the graph $G$.
The \emph{closed neighborhood} of $u$ is defined by $N[u] = N(u)\cup\{u\}$.
The \emph{degree} of $u$ is defined by $\deg(u) = \abs{N_{G}(u)}$.
When $\deg(u)=1$ we refer to $u$ as a \emph{leaf}.

The graph $G'=(V',E')$ is a \emph{subgraph} of $G$ if $V'\subseteq V$ and $E'\subseteq E$. 
Given $W\subseteq V$, the subgraph induced by $W$ is the subgraph of $G$ made up of the vertices in $W$ and the edges $\{u,v\}\in E$ such that $u,v\in W$. 
We reference this subgraph as an \emph{induced subgraph} and denote it by $G[W]$. 
When $G[W]$ is a complete graph, we reference $W$ as a \emph{clique}. 
The \emph{clique cover number}, denoted $\cc(G)$, is the minimum number of cliques needed to cover the entire graph.

Given vertices $u,v\in V$, we say that $u$ is \emph{connected to} $v$ if there exists a list of distinct vertices $(w_{0},w_{1},\ldots,w_{l})$ such that $u=w_{0}$, $v=w_{l}$, $\{w_{i},w_{i+1}\}\in E$ for all $i=0,1,\ldots,l-1$, and $w_{i}\neq w_{j}$ for all $i\neq j$. 
Such a list of vertices is called a \emph{path}; in particular, we reference it as a $(u,v)$-path. 
We say that a graph $G=(V,E)$ is \emph{connected} if for every pair of vertices $u,v\in V$ there exists a $(u,v)$-path; otherwise, we call $G$ \emph{disconnected}.
Given $u\in V$, the graph $G-u$ is the induced subgraph $G[V\setminus\{u\}]$ obtained by deleting the vertex $u$ from the graph $G$. 
A graph is $k$-connected if deleting fewer than $k$ vertices does not disconnect the graph, that is, there is at least $k$ distinct $(u,v)$-paths between any two vertices $u,v\in V$.

A \emph{path graph} of order $n$, denoted $P_{n}$, has vertex set $V=\{v_{1},v_{2},\ldots,v_{n}\}$ and edge set 
\[
E = \left\{\{v_{i},v_{i+1}\}\colon 1\leq i\leq n-1\right\}.
\]
A \emph{cycle graph} of order $n$, denoted $C_{n}$, can be obtained from $P_{n}$ by adding the edge $\{v_{1},v_{n}\}$. 
We say that $G=(V,E)$ is a \emph{tree} if $G$ is connected and $G[W]$ is not a cycle for any $W\subseteq V$.
The \emph{star graph} of order $n$ is a tree with a single central vertex adjacent to $(n-1)$ leafs. 
The \emph{hypercube graph} of dimension $d$, denoted $Q_{d}$, can be defined recursively as $Q_{1}=P_{2}$ and
\[
Q_{i} = Q_{i-1}\Box Q_{1},~2\leq i\leq d,
\]
where $\Box$ denotes the Cartesian product.

In general, given graphs $G=(V,E)$ and $G'=(V',E')$ in $\mathbb{G}$ such that $V\cap V'=\emptyset$, the \emph{Cartesian product} $G\Box G'$ has vertex set $V(G\Box G') = V(G)\times V(G')$ and edge set
\[
E(G\Box G') = \left\{\{(u,u'),(v,v')\}\colon u=v,~u'\in N_{G'}(v')~\textrm{or}~u'=v',~u\in N_{G}(v)\right\}.
\]
In addition, the \emph{vertex sum} of $G$ and $G'$, denoted $G\oplus_{v}G'$, is the graph formed by identifying vertices $u\in V$ and $u'\in V'$ to form the vertex $v$.
That is, $G\oplus_{v}G'$ has vertex set
\[
V(G\oplus_{v}G') = \left(V\setminus\{u\}\right)\cup\left(V'\setminus\{u'\}\right)\cup\{v\}
\]
and edge set
\[
E(G\oplus_{v}G') = \left(E\setminus{S_{u}}\right)\cup\left(E'\setminus{S_{u'}}\right)\cup S_{v},
\]
where $S_{u}$ denotes all edges of $G$ that are incident on $u$, $S_{u'}$ denotes all edges of $G'$ that are incident on $u'$, and $S_{v}=\left\{\{v,w\}\colon w\in N_{G}(u)\cup N_{G'}(u')\right\}$.
The \emph{edge sum} of $G$ and $G'$, denoted $G+_{e}G'$, is the graph formed by identifying $u\in V$ and $u'\in V'$ to form the edge $e=\{u,u'\}$.
That is, $G+_{e}G'$ has vertex set
\[
V(G+_{e}G') = V\cup V'
\]
and edge set 
\[
E(G+_{e}G') = E\cup E'\cup\{u,u'\}. 
\]
\subsection{Zero forcing}\label{subsec:prelim_zero_forcing}
\emph{Zero forcing} is a binary coloring game on a graph, where vertices are either filled or non-filled. 
In this article, we denote filled vertices by the color gray and non-filled vertices by the color white. 
An initial set of gray vertices can force white vertices to become gray following a color change rule. 
While there are many color change rules, see~\cite[Chapter 9]{Hogben2022}, we use the \emph{standard rule} which states that a gray vertex $u$ can force a white vertex $v$ if $v$ is the only white neighbor of $u$.

A zero forcing game on $G$ corresponds to a collection of subsets $C^{(t)}$ and $C^{[t]}$, $t\geq 0$, and a collection of forces $\phi(C)$.
In particular, $C=C^{(0)}=C^{[0]}$ is the set of initially filled vertices, $C^{(t)}$ denotes the set of vertices that are forced at time step $t$, and $C^{[t]}$ denotes the set of all filled vertices at time step $t$; hence, $C^{[t]} = C^{[t-1]}\cup C^{(t)},~t\geq 0$.
Furthermore, every vertex of $G$ is in exactly one set $C^{(t)}$, and if $v\in C^{(t)}$ then $v$ must be forced by exactly one of its neighbors $u$ such that $u$ and all of its neighbors except for $v$ are in $C^{[t-1]}$.
We denote such a forcing by $u\rightarrow v$.
The collection $\phi(C)$ stores all such forcings in the zero forcing game.

Assuming the graph is finite and at least one of the potential forces are applied at each time step, there exists a time step $t^{*}\geq 0$ for which no more forcings can be applied; hence, $C^{[t^{*}]}=C^{[t^{*}+t]}$, for all $t\geq 0$.
We reference $C^{[t^{*}]}$ as the \emph{closure} of $C$ and denote it by $\cl{C}$.
Furthermore, we say that $\phi(C)$ is a collection of forces that produce $\cl{C}$.
Note that $\phi(C)$ may not be unique since a vertex may be forced by several of its neighbors, but only one forcing will be used in the zero forcing game. 
If $\cl{C}=V$, then we say that $C$ is a \emph{zero forcing set} of $G$.
The \emph{zero forcing number} of $G$ is defined as
\[
\Z(G) = \min\left\{\abs{C}\colon\cl{C}=V\right\}.
\]
If $C$ is a zero forcing set of $G$ such that $\Z(G)=\abs{C}$, we say that $C$ is a \emph{minimum zero forcing set}.

The \emph{propagation time} of $C$, denoted $\pt(G,C)$, is the smallest $t^{*}$ such that $C^{[t^{*}]}=V$ or $\infty$ if $C$ is not a zero forcing set of $G$. 
When measuring propagation time, it is assumed that in a single time step all possible forces that can be done independently of each other are applied.
The \emph{minimum propagation time} of $G$ is defined as
\[
\pt(G) = \min\left\{\pt(G,C)\colon\cl{C}=V,~\abs{C}=Z(G)\right\}.
\]
and the \emph{maximum propagation time} of $G$ is defined as
\[
\PT(G) = \max\left\{\pt(G,C)\colon\cl{C}=V,~\abs{C}=Z(G)\right\}.
\]
In addition, the \emph{throttling number} of $C$ is defined as $\th(G,C) = \abs{C} + \pt(G,C)$ and the throttling number of $G$ is defined as
\[
\th(G) = \min_{C\subseteq V}\th(G,C).
\]
To our knowledge, no where else in literature has considered the computational complexity of the propagation time.
We define the decision version of the propagation time as follows: Given a graph $G$ and integers $s$ and $z$, is there is a zero forcing set of $G$ with cardinality less than $s$ and propagation time less than $z$?
It is clear that this decision problem is in NP since given a zero forcing set we can arrive at its closure in polynomial time, see Algorithm 1 in~\cite{Brimkov2019zf}.
Moreover, we can trivially reduce the decision version of the zero forcing problem to the propagation time decision problem. Indeed, the decision version of the zero forcing problem asks the following: Given a graph $G$ and integer $k$, is there a zero forcing set of $G$ with cardinality less than $k$? Setting $s=k$ and $z=n-k+1$, we reduce the zero forcing problem to the propagation time problem. Hence, the propagation time decision problem is NP-complete, which implies that the combinatorial optimization problem for both the minimum and maximum propagation time is NP-hard.

The \emph{propagation time interval} of $G$ is defined as
\[
\left[\pt(G),\PT(G)\right] = \left\{\pt(G),\pt(G)+1,\ldots,\PT(G)-1,\PT(G)\right\}.
\]
The \emph{realized propagation time interval} is the set of all $k\in\left[\pt(G),\PT(G)\right]$ such that there is a minimum zero forcing set $C$ that realizes $\pt(G,C)=k$. 
The graph $G$ has a \emph{full propagation time interval} if the realized propagation time interval is equal to $\left[\pt(G),\PT(G)\right]$.
There are only a few families of graphs that are known to have a full propagation time interval, for example, the path, cycle, and complete bipartite graphs~\cite{Hogben2012,Hogben2022}. 
In~\cite{Brimkov2025}, the authors make the following conjecture regarding the propagation time interval of the hypercube graph.
\begin{conjecture}[{\cite[Conjecture 4.4]{Brimkov2025}}]\label{con:pti_hypercube}
Let $d\geq 2$.
Then, $\pt(Q_{d}) = 1$ and $\PT(Q_{d})=2^{d-2}$.
Moreover, $Q_{d}$ has a full propagation time interval. 
\end{conjecture}
\subsection{Forts}\label{subsec:prelim_forts}
Given a graph $G\in\mathbb{G}$, a non-empty subset $F\subseteq V$ is a \emph{fort} if no vertex $u\in V\setminus{F}$ has exactly one neighbor in $F$.
Let $\mathcal{F}_{G}$ denote the collection of all forts of $G$. 
Note that every zero forcing set of $G$ forms a cover for $\mathcal{F}_{G}$; in particular, Theorem~\ref{thm:fort_cover} states that a subset of vertices is a zero forcing set if and only if that subset intersects every fort. 
While one direction of this result was originally proven in~\cite[Theorem 3]{Fast2018}, both directions are shown in~\cite[Theorem 8]{Brimkov2019zf}.
\begin{theorem}[{\cite[Theorem 8]{Brimkov2019zf}}]\label{thm:fort_cover}
Let $G\in\mathbb{G}$.
Then, $S\subseteq V$ is a zero forcing set of $G$ if and only if $S$ intersects every fort in $\mathcal{F}_{G}$.
\end{theorem}

Theorem~\ref{thm:fort_cover} motivates the Fort Cover Model, see~\eqref{eq:fc_obj}--\eqref{eq:fc_const2}, which models zero forcing as a set cover problem. 
Moreover, the optimal value of the linear programming relaxation of the Fort Cover Model is the fractional zero forcing number, denoted $\Z^{*}(G)$, as defined in~\cite{Cameron2023}. 
The fort number of a graph is defined by
\[
\ft(G) = \max\left\{\abs{\mathcal{F}}\colon\textrm{$\mathcal{F}$ is a collection of disjoint forts of $G$}\right\}
\]
and the following inequality holds for all graphs
\[
\ft(G) \leq \Z^{*}(G) \leq \Z(G),
\]
see~\cite[Remark 2.5]{Cameron2023}. 

The minimal forts of a graph play an important role in the facet defining inequalities of the Fort Cover Model, see~\cite[Theorem 10]{Brimkov2019zf} and~\cite[Theorem 1]{Brimkov2021}.
In~\cite{Becker2025}, the authors enumerate the number of minimal forts for several families of graphs and identify families with an exponential number of minimal forts.
Moreover, they make the following conjecture regarding the maximum number of minimal forts of a tree. 
\begin{conjecture}[{\cite[Conjecture 30]{Becker2025}}]\label{con:tree_gr}
Let $T_{n}$ denote a tree with the largest number of minimal forts over all trees of order $n\geq 1$. 
Then, the number of minimal forts satisfies
\[
\abs{\mathcal{F}_{T_{n}}} \leq \binom{n}{2}\abs{\mathcal{F}_{P_{n}}}.
\]
Moreover, the limit of the ratio of consecutive terms satisfies
\[
\lim_{n\rightarrow\infty}\frac{\abs{\mathcal{F}_{T_{n+1}}}}{\abs{\mathcal{F}_{T_{n}}}} = \lim_{n\rightarrow\infty}\frac{\abs{\mathcal{F}_{P_{n+1}}}}{\abs{\mathcal{F}_{P_{n}}}} = \psi,
\]
where $\psi=1.3247179572\ldots$ denotes the plastic ratio.
\end{conjecture}
\subsection{Maximum Nullity}\label{subsec:prelim_max_nullity}
Given a graph $G\in\mathbb{G}$, let $S(G)$ denote the set of all symmetric matrices $A$ such that $a_{ij}\neq 0$ if and only if $\{i,j\}\in E$.
There are no conditions on the diagonal entries of $A$, which allows for diagonal translation.
Hence, the maximum multiplicity of an eigenvalue among all $A\in S(G)$ can be viewed as the maximum nullity among all $A\in S(G)$.
The \emph{maximum nullity} of $G$, denoted $\M(G)$, is the maximum nullity among all $A\in S(G)$ and is an important relaxation of the inverse eigenvalue problem~\cite{Hogben2022}. 
A related parameter is the \emph{minimum rank} of $G$, denoted $\mr(G)$, which is defined as the minimum rank among all $A\in S(G)$.
By the rank-nullity theorem, see~\cite[Section 0.4.4]{Horn2012}, it is clear that $\mr(G)=n-\M(G)$. 

In~\cite{AIM2008}, the authors show that $\M(G)\leq\Z(G)$ for all graphs $G$. 
Moreover, for graphs of order six or less, it was shown that $\M(G)=\Z(G)$.
These results generated a lot of interest and were extended in~\cite{DeLoss2008}, where the authors compute the minimum rank of all graphs of order seven.
By examination, it can be seen that $\M(G)=\Z(G)$ for all graphs of order seven or less. 
More recently, the maximum nullity for all graphs of order eight was determined~\cite{Barrett2025}.
In particular, the authors identify seven graphs $E_{1},E_{2},\ldots,E_{7}$ of order eight such that the following result holds.
\begin{theorem}[{\cite[Theorem 1.3]{Barrett2025}}]\label{thm:max_nullity8}
Let $\mathcal{E} = \left\{E_{1},E_{2},\ldots,E_{7}\right\}$.
Then, any graph $G$ of order eight satisfies $M(G)<Z(G)$ if and only if $G\in\mathcal{E}$.
\end{theorem}

The graphs in Theorem~\ref{thm:max_nullity8} are the smallest order graphs for which there is a gap between the maximum nullity and the zero forcing number. 
In particular, $M(E_{i})=Z(E_{i})-1$, for all $1\leq i\leq 7$. 
Interestingly, it is known that $\M(G) < \Z(G)$ for most graphs; indeed, $\Z(G)-\M(G)\geq 0.14n$ for almost all graphs of order $n$ sufficiently large~\cite{Ekstrand2013,Hall2010}.
This observation motivates interest in lower bounds on the maximum nullity of a graph. 
In~\cite{Holst2008}, it is shown that every $k$-connected graph has a maximum nullity of at least $k$, in~\cite{Fallat2007} it is shown that $\M(G)\geq n - \cc(G)$, and in~\cite{Johnson1999} it is shown that $\M(G)\geq \max\{p-q\}$ such that deleting $q$ vertices splits $G$ into $p$ paths. 
In~\cite{Cameron2023}, the following lower bounds are conjectured for the maximum nullity of a graph.
\begin{conjecture}[{\cite[Question 6.1]{Cameron2023}}]\label{con:mn_lower_bounds}
Let $G\in\mathbb{G}$.
Then,
\[
\ft(G) \leq \Z^{*}(G) \leq \M(G).
\]
\end{conjecture}
\section{Infection Model}\label{sec:infection}
The Infection Model was first proposed in~\cite{Brimkov2019zf} for computing the zero forcing number of a graph. 
In this section, we introduce a similar model with various objective functions for computing the zero forcing number, minimum propagation time, and throttling number of a graph. 

Given $G\in\mathbb{G}$, let $A$ denote the set of antiparallel arcs, that is, for each $\{u,v\}\in E$, we have $(u,v)\in A$ and $(v,u)\in A$.
Also, let $T$ denote the maximum number of time steps necessary to reach the closure, given any zero forcing set.
In general, $T$ can be set to $(n-1)$ but there may be sharper bounds that can be given for specific graphs. 
We denote by $\im{G}{T}$ the Infection Model for graph $G$ and maximum number of time steps $T$, which is shown in~\eqref{eq:im_obj}--\eqref{eq:im_const8}.
Note that the original Infection Model from~\cite{Brimkov2019zf} did not include constraint~\eqref{eq:im_const4} nor the integer variable $z$. 
Here, we let $[T]=\{1,\ldots,T\}$.
Also, the binary variable $s_{v}$ indicates whether the vertex $v$ is in the zero forcing set; the integer variable $x_{v}$ indicates at what time step the vertex $v$ is forced; the binary variable $y_{a}$, where $a=(u,v)\in A$, indicates whether $u$ forces $v$; the integer variable $z$ provides an upper bound on the propagation time of the zero forcing set.

\begin{mini!}
	{}{\sum_{v\in V}s_{v}}{}{}\label{eq:im_obj}
	\addConstraint{s_{v}+\sum_{a=(u,v)\in A}y_{a}}{=1,~\quad\forall v\in V}\label{eq:im_const1}
	\addConstraint{x_{u}-x_{v}+(T+1)y_{a}}{\leq T,~\quad\forall a=(u,v)\in A}\label{eq:im_const2}
	\addConstraint{x_{w}-x_{v}+(T+1)y_{a}}{\leq T,~\quad\forall a=(u,v)\in A,~w\in N(u)\setminus\{v\}}\label{eq:im_const3}
	\addConstraint{x_{v}-z}{\leq 0,~\quad\forall v\in V}\label{eq:im_const4}
	\addConstraint{s_{v}}{\in\{0,1\},~\quad\forall v\in V}\label{eq:im_const5}
	\addConstraint{x_{v}}{\in[T]\cup\{0\},~\quad\forall v\in V}\label{eq:im_const6}
	\addConstraint{y_{a}}{\in\{0,1\},~\quad\forall a\in A}\label{eq:im_const7}
	\addConstraint{z}{\in[T]\cup\{0\}}\label{eq:im_const8}
\end{mini!}
\subsection{Feasible solutions}\label{subsec:infection_feasible}
We say that $(s,x,y,z)$ is a feasible solution of $\im{G}{T}$ if constraints~\eqref{eq:im_const1}--\eqref{eq:im_const8} are satisfied. 
Furthermore, we denote by $\feas{\im{G}{T}}$ the set of all feasible solutions of $\im{G}{T}$.
The following result shows that the feasible solutions of $\im{G}{T}$ correspond to a zero forcing game on the graph $G$, where the initial set of filled vertices is a zero forcing set and no more than $T$ time steps are used.
Recall from Section~\ref{subsec:prelim_zero_forcing}, we represent a zero forcing game by a collection of subsets $C^{(t)}$ and $C^{[t]}$, where $0\leq t\leq T$, and a collection of forces $\phi(C)$.
In particular, where $C=C^{(0)}=C^{[0]}$ is the set of initially filled vertices, $C^{(t)}$ denotes the set of vertices that are forced at time step $t$, $C^{[t]}$ denotes the set of all filled vertices at time step $t$, and $\phi(C)$ denotes the forces that produce $\cl{C}$. 
We denote by $\mathcal{Z}(G,T)$ the family of all zero forcing games on the graph $G$, where the initial filling is a zero forcing set and no more than $T$ time steps are used. 
\begin{theorem}\label{thm:im_feas1}
Let $G\in\mathbb{G}$.
If $(s,x,y,z)\in\feas{\im{G}{T}}$, then $C=\left\{v\in V\colon s_{v}=1\right\}$ is a zero forcing set corresponding to a zero forcing game in $\mathcal{Z}(G,T)$, where $\pt(G,C)\leq z\leq T$.
\end{theorem}
\begin{proof}
Let $(s,x,y,z)\in\feas{\im{G}{T}}$.
Then, define $C=C^{(0)} = C^{[0]} = \left\{v\in V\colon s_{v}=1\right\}$ and
\begin{align*}
C^{(t)} &= \left\{v\in V\colon x_{v}=t\right\} \\
C^{[t]} &= C^{[t-1]}\cup C^{(t)},~1\leq t\leq T.
\end{align*}
Further, define $\phi(C)$ by the collection of forces $u\rightarrow v$ such that $y_{a}=1$, where $a=(u,v)$.
Note that every vertex of $G$ is in exactly one set $C^{(t)}$, where $0\leq t\leq T$.
Also, by constraint~\eqref{eq:im_const1}, every $v\in V$ is either in $C$ or corresponds to exactly one arc $a=(u,v)$ such that $y_{a}=1$.
If $y_{a}=1$, then constraints~\eqref{eq:im_const2}--\eqref{eq:im_const3} imply that
\[
x_{v} \geq 1 + x_{u}~\textrm{and}~x_{v}\geq 1+x_{w},
\]
for all $w\in N(u)\setminus\{v\}$.
Hence, $v\in C^{(t)}$ for some $t\in[T]$ and
\[
x_{u}\leq t-1~\textrm{and}~x_{w}\leq t-1,
\]
for all $w\in N(u)\setminus\{v\}$, which implies that $u$ and all of its neighbors except for $v$ are in $C^{[t-1]}$.
Finally, constraint~\eqref{eq:im_const4} implies that the closure of $C$ occurs in at most $z$ time steps. 

Therefore, every vertex of $G$ is either initially filled in $C$ or is forced later in $C^{(t)}$, for exactly one $t\in[T]$.
Furthermore, if $v\in C^{(t)}$, then $v$ must be forced by exactly one if its neighbors $u$ such that $u$ and all of its neighbors except for $v$ are in $C^{[t-1]}$, and the forcing $u\rightarrow v$ is stored in $\phi(C)$. 
It follows that $C$ is a zero forcing set whose closure occurs in at most $z\leq T$ time steps, and $\phi(C)$ is a collection of forces that produce $\cl{C}$.
\end{proof}

Next, we show that every zero forcing game in $\mathcal{Z}(G,T)$ corresponds to a feasible solution of $\im{G}{T}$.
\begin{theorem}\label{thm:im_feas2}
Let $G\in\mathbb{G}$.
For each zero forcing set $C$ corresponding to a zero forcing game in $\mathcal{Z}(G,T)$, there is a feasible solution $(s,x,y,z)\in\feas{\im{G}{T}}$ such that $\abs{C} = \sum_{v\in V}s_{v}$ and $z=\pt(G,C)\leq T$.
\end{theorem}
\begin{proof}
Consider the collection of subsets $C^{(t)}$ and $C^{[t]}$, where $0\leq t\leq T$, and the collection of forces $\phi(C)$ that represent the given zero forcing game.
For each $v\in V$, define $s_{v}=1$ if $v\in C=C^{(0)}=C^{[0]}$; otherwise, set $s_{v}=0$. 
Also, for each $v\in V$, define $x_{v}=t$ if $v\in C^{(t)}$.
For each $a=(u,v)\in A$, define $y_{a}=1$  if $u\rightarrow v$ is in $\phi(C)$; otherwise, set $y_{a}=0$.
Finally, set $z$ to the time step at which the closure of $C$ occurs. 

For every $v\in V$, $v\in C$ or there is exactly one neighbor $u$ such that $u\rightarrow v$ is in $\phi(C)$. 
Hence, constraint~\eqref{eq:im_const1} holds.
Furthermore, if $u\rightarrow v$ is in $\phi(C)$, then $v\in C^{(t)}$ for exactly one $t\in[T]$ and $u$ and all of its neighbors except for $v$ are in $C^{[t-1]}$.
Hence, constraints~\eqref{eq:im_const2}--\eqref{eq:im_const3} hold. 
Since $z$ is set to the time step at which the closure of $C$ occurs, it follows that constraint~\eqref{eq:im_const4} holds. 
\end{proof}
\subsection{Optimal solutions}\label{subsec:infection_optimal}
We say that $(s,x,y,z)\in\feas{\im{G}{T}}$ is an optimal solution if it minimizes the objective function in~\eqref{eq:im_obj}.
We denote by $\opt{\im{G}{T}}$ the set of all optimal solutions of $\im{G}{T}$.
The following result shows that the optimal solutions of $\im{G}{T}$ correspond to minimum zero forcing sets of $G$. 
\begin{corollary}\label{cor:im_opt}
Let $G\in\mathbb{G}$ and $T=(n-1)$. 
If $(s,x,y,z)\in\opt{\im{G}{T}}$, then $C=\left\{v\in V\colon s_{v}=1\right\}$ is a minimum zero forcing set of $G$.
\end{corollary}
\begin{proof}
Let $(s,x,y,z)\in\opt{\im{G}{T}}$.
Then, $(s,x,y,z)\in\feas{\im{G}{T}}$ and is minimizes the objective function in~\eqref{eq:im_obj}.
Define $C=\left\{v\in V\colon s_{v}=1\right\}$. 
Since $(s,x,y,z)$ is a feasible solution, Theorem~\ref{thm:im_feas1} implies that $C$ is a zero forcing set corresponding to a zero forcing game in $\mathcal{Z}(G,T)$.
Therefore, $\Z(G)\leq \abs{C}$.
For the sake of contradiction, suppose that $\Z(G) < \abs{C}$.
Then, there is a zero forcing  set $\hat{C}$ corresponding to a zero forcing game in $\mathcal{Z}(G,T)$ such that
\[
\abs{\hat{C}} = \Z(G) < \abs{C}
\]
However, Theorem~\ref{thm:im_feas2} implies that there exists a feasible solution $(\hat{s},\hat{x},\hat{y},\hat{z})\in\feas{\im{G}{T}}$ such that $\abs{\hat{C}}=\sum_{v\in V}\hat{s}_{v}$, which contradicts the optimality of $(s,x,y,z)\in\opt{\im{G}{T}}$.
\end{proof}

In what follows, we make use of two other objective functions to compute the minimum propagation time and throttling number of a graph.
Specifically, let $\impt{G}{T}$ denote the model in~\eqref{eq:im_obj}--\eqref{eq:im_const8}, where~\eqref{eq:im_obj} is replaced by the objective function
\begin{equation}\label{eq:impt_obj}
\textrm{minimize}~\sum_{v\in V}s_{v} + \frac{z}{2T}.
\end{equation}
Also, let $\imth{G}{T}$ denote the model in~\eqref{eq:im_obj}--\eqref{eq:im_const8}, where~\eqref{eq:im_obj} is replaced by the objective function
\begin{equation}\label{eq:imth_obj}
\textrm{minimize}~\sum_{v\in V}s_{v} + z.
\end{equation}
Note that $\im{G}{T}$, $\impt{G}{T}$, and $\imth{G}{T}$ all have the same feasible solutions but potentially different optimal solutions. 
The following result shows that the optimal solutions of $\impt{G}{T}$ correspond to minimum zero forcing sets with minimum propagation time. 
\begin{corollary}\label{cor:impt_opt}
Let $G\in\mathbb{G}$ and $T=(n-1)$. 
If $(s,x,y,z)\in\opt{\impt{G}{T}}$, then $C=\left\{v\in V\colon s_{v}=1\right\}$ is a minimum zero forcing set of $G$ such that $\pt(G)=\pt(G,C)=z$.
\end{corollary}
\begin{proof}
Let $(s,x,y,z)\in\opt{\impt{G}{T}}$.
Then, $(s,x,y,z)\in\feas{\im{G}{T}}$ and is minimizes the objective function in~\eqref{eq:impt_obj}.
Define $C=\left\{v\in V\colon s_{v}=1\right\}$. 
Since $0\leq \frac{z}{2T}\leq \frac{1}{2}$, it follows that the objective function first minimizes $\sum_{v\in V}s_{v}$ and then minimizes $\frac{z}{2T}$.
Hence, by Corollary~\ref{cor:im_opt}, $C$ is a minimum zero forcing set of $G$.
Furthermore, since $(s,x,y,z)$ is a feasible solution, Theorem~\ref{thm:im_feas1} implies that
\[
\pt(G)\leq \pt(G,C)\leq z. 
\]
For the sake of contradiction, suppose that $\pt(G)<z$. 
Then, there is a zero forcing set $\hat{C}$ corresponding to a zero forcing game in $\mathcal{Z}(G,T)$ such that 
\[
\pt(G,\hat{C})=\pt(G) < z
~\text{and}~
\abs{\hat{C}} = \Z(G).
\]
However, Theorem~\ref{thm:im_feas2} implies that there exists a feasible solution $(\hat{s},\hat{x},\hat{y},\hat{z})\in\feas{\im{G}{T}}$ such that $\hat{z}=\pt(G,\hat{C})$ and $\Z(G)=\sum_{v\in V}\hat{s}_{v}$, which contradicts the optimality of $(s,x,y,z)\in\opt{\impt{G}{T}}$.
\end{proof}
The following result shows that the optimal solutions of $\imth{G}{T}$ correspond to zero forcing sets with minimum throttling number. 
\begin{corollary}\label{cor:imth_opt}
Let $G\in\mathbb{G}$ and $T=(n-1)$.
If $(s,x,y,z)\in\opt{\imth{G}{T}}$, then $C=\left\{v\in V\colon s_{v}=1\right\}$ is a zero forcing set of $G$ such that $\th(G)=\th(G,C) = \sum_{v\in V}s_{v}+z$.
\end{corollary}
\begin{proof}
Let $(s,x,y,z)\in\opt{\imth{G}{T}}$.
Then, $(s,x,y,z)\in\feas{\im{G}{T}}$ and is minimizes the objective function in~\eqref{eq:imth_obj}.
Define $C=\left\{v\in V\colon s_{v}=1\right\}$.
Since $(s,x,y,z)$ is a feasible solution, Theorem~\ref{thm:im_feas1} implies that $C$ is a zero forcing set corresponding to a zero forcing game in $\mathcal{Z}(G,T)$, where $\pt(G,C)\leq z\leq T$. 
Therefore,
\[
\th(G)\leq \th(G,C) \leq \sum_{v\in V}s_{v} + z.
\]
For the sake of contradiction, suppose that $\th(G) < \sum_{v\in V}s_{v} + z$.
Then, there exists a zero forcing set $\hat{C}$ corresponding to a zero forcing game in $\mathcal{Z}(G,T)$ such that 
\[
\th(G,\hat{C}) = \th(G) < \sum_{v\in V}s_{v} + z.
\]
However, Theorem~\ref{thm:im_feas2} implies that there exists a feasible solution $(\hat{s},\hat{x},\hat{y},\hat{z})\in\feas{\im{G}{T}}$ such that $\abs{\hat{C}} = \sum_{v\in V}\hat{s}_{v}$ and $\pt(G,\hat{C}) = \hat{z}$, which contradicts the optimality of $(s,x,y,z)\in\opt{\imth{G}{T}}$.
\end{proof}
\section{Time Step Model}\label{sec:time_step}
The Time Step Model was first proposed in~\cite{Agra2019} for computing the zero forcing number of a graph. 
In this section, we introduce a similar model with various objective functions for computing the zero forcing number, minimum and maximum propagation times, and throttling number of a graph. 
Note that the Infection Model in~\eqref{eq:im_obj}--\eqref{eq:im_const8} is incapable of computing parameters such as the maximum propagation time since there are feasible solutions that correspond to zero forcing games where not all possible forces are applied in a single time step, see Example~\ref{ex:wrong_PT}. 
As we will see, this is not the case for our version of the Time Step Model.
\begin{example}\label{ex:wrong_PT}
Let $G$ denote the tree shown in Figure~\ref{fig:t6}.
The following is a feasible solution of $\im{G}{5}$:
\[
s = \{1,0,0,1,0,0\},~x = \{0,1,2,0,3,5\},~y = \{1,0,1,0,0,0,1,0,1,0\},
\]
where the indices of $y$ correspond to the arcs $(1,2)$, $(2,1)$, $(2,3)$, $(3,2)$, $(3,4)$, $(4,3)$, $(3,5)$, $(5,3)$, $(5,6)$, and $(6,5)$, respectively. 
In the first time step, the forcing $1\rightarrow 2$ occurs, but not $4\rightarrow 3$. 
In the second step, $2\rightarrow 3$ occurs. 
In the third, $3\rightarrow 5$ occurs. 
In the fourth, despite the possible forcing $5\rightarrow 6$, nothing occurs. 
Finally, in the fifth step, $5\rightarrow 6$ occurs. 
This suggests a maximum propagation time of $5$, but it is actually $3$.,

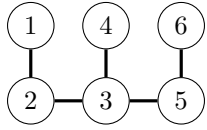
\begin{figure}[ht!]
\centering 
\begin{tikzpicture}
[nodeFilledDecorate/.style={shape=circle,inner sep=2pt,draw=black,fill=lightgray,thick},%
    nodeEmptyDecorate/.style={shape=circle,inner sep=2pt,draw=black,thick},%
    lineDecorate/.style={black,=>latex',very thick},%
    x=1cm,y=1cm,scale=1.0]
    \node[circle,draw=black,fill=white] (1) at (0,1) {$1$};
    \node[circle,draw=black,fill=white] (2) at (0,0) {$2$};
    \node[circle,draw=black,fill=white] (3) at (1,0) {$3$};
    \node[circle,draw=black,fill=white] (4) at (1,1) {$4$};
    \node[circle,draw=black,fill=white] (5) at (2,0) {$5$};
    \node[circle,draw=black,fill=white] (6) at (2,1) {$6$};

    \foreach \x/\y in {1/2,2/3,3/4,3/5,5/6}
        \draw[black,=>latex',-,very thick] (\x) -- (\y);
\end{tikzpicture}
\caption{A tree of order $6$.}
\label{fig:t6}
\end{figure}
\end{example}

Given $G\in\mathbb{G}$, let $A$ denote the set of antiparallel arcs, that is, for each $\{u,v\}\in E$, we have $(u,v)\in A$ and $(v,u)\in A$.
Also, let $T$ denote the maximum number of time steps necessary to reach the closure, given any zero forcing set.
In general, $T$ can be set to $(n-1)$ but there may be sharper bounds that can be given for specific graphs. 
We denote by $\tsm{G}{T}$ the Time Step Model for graph $G$ and maximum number of time steps $T$, which is shown in~\eqref{eq:tsm_obj}--\eqref{eq:tsm_const10}.
Note that the original Time Step Model from~\cite{Agra2019} did not include constraints~\eqref{eq:tsm_const5}--\eqref{eq:tsm_const7} nor the binary variable $z^{t}$. 
Here, we let $[T]=\{1,\ldots,T\}$.
Also, the binary variable $x^{t}_{v}$ indicates that $v$ is forced by time step $t$; the binary variable $y^{t}_{a}$, where $a=(u,v)\in A$, indicates whether $u$ forces $v$ at time step $t$; the binary variable $z^{t}$ indicates if there are any forces that occur at time step $t$.

\begin{mini!}
	{}{\sum_{v\in V}x^{0}_{v}}{}{}\label{eq:tsm_obj}
	\addConstraint{x^{0}_{v}+\sum_{t\in[T]}\sum_{a=(u,v)\in A}y^{t}_{a}}{=1,~\quad\forall v\in V}\label{eq:tsm_const1}
	\addConstraint{y^{t}_{a}}{\leq x^{t-1}_{u},~\quad\forall a=(u,v)\in A,~t\in[T]}\label{eq:tsm_const2}
	\addConstraint{y^{t}_{a}}{\leq x^{t-1}_{w},~\quad\forall a=(u,v)\in A,~w\in N(u)\setminus\{v\},~t\in[T]}\label{eq:tsm_const3}
	\addConstraint{x^{t}_{v}}{=x^{t-1}_{v} + \sum_{a=(u,v)\in A}y^{t}_{a},~\quad\forall v\in V,~t\in[T]}\label{eq:tsm_const4}
    \addConstraint{x^{t-1}_{u}-x^{t-1}_{v} + \sum_{w\in N(u)\setminus\{v\}}x^{t-1}_{w}}{\leq\sum_{a=(w,v)\in A}y^{t}_{a} + d(u)-1,~\quad\forall (u,v)\in A,~t\in[T]}\label{eq:tsm_const5}
    \addConstraint{\frac{1}{n}\sum_{v\in V}\left(x^{t}_{v}-x^{t-1}_{v}\right)-z^{t}}{\leq 0,~\quad\forall t\in[T]}\label{eq:tsm_const6}
    \addConstraint{z^{t}-\sum_{v\in V}\left(x^{t}_{v}-x^{t-1}_{v}\right)}{\leq 0,~\quad\forall t\in[T]}\label{eq:tsm_const7}
	\addConstraint{x^{t}_{v}}{\in\{0,1\},~\forall v\in V,~t\in[T]\cup\{0\}}\label{eq:tsm_const8}
	\addConstraint{y^{t}_{a}}{\in\{0,1\},~\forall a\in A,~t\in[T]}\label{eq:tsm_const9}
	\addConstraint{z^{t}}{\in\{0,1\},~\forall t\in[T]}\label{eq:tsm_const10}
\end{mini!}
\subsection{Feasible solutions}\label{subsec:time_step_feasible}
We say that $(s,x,y,z)$ is a feasible solution of $\tsm{G}{T}$ if constraints~\eqref{eq:tsm_const1}--\eqref{eq:tsm_const10} are satisfied. 
Furthermore, we denote by $\feas{\tsm{G}{T}}$ the set of all feasible solutions of $\tsm{G}{T}$.
The following result shows that each feasible solution of $\tsm{G}{T}$ corresponds to a zero forcing game on the graph $G$, where the initial filling is a zero forcing set, no more than $T$ time steps are used, and on each time step all possible forces that can be done independently of each other are applied. 
To that end, we represent a zero forcing game by a collection of subsets $C^{(t)}$ and $C^{[t]}$, where $0\leq t\leq T$, and a collection of forces $\phi(C)$.
In contrast to the general rule given in Section~\ref{subsec:prelim_zero_forcing}, here $C^{(t)}$ denotes the set of all vertices that can be independently forced at time step $t$.
We denote by $\mathcal{Z}(G,T)$ the family of all zero forcing games on the graph $G$, where the initial filling is a zero forcing set, no more than $T$ time steps are used, and on each time step all possible forces that can be done independently of each other are applied.
\begin{theorem}\label{thm:tsm_feas1}
Let $G\in\mathbb{G}$.
If $(x,y,z)\in\feas{\tsm{G}{T}}$, then $C=\left\{v\in V\colon x^{0}_{v}=1\right\}$ is a zero forcing set corresponding to a zero forcing game in $\mathcal{Z}(G,T)$, where $\pt(G,C)=\sum_{t\in[T]}z^{t}$.
\end{theorem}
\begin{proof}
Let $(x,y,z)\in\feas{\tsm{G}{T}}$. 
Then, define $C=C^{(0)}=C^{[0]}=\{v\in V\colon x_{v}^{0}=1\}$ and
\begin{align*}
    C^{(t)} = \left\{v\in V\colon x_{v}^{t}=1,~v\notin C^{[t-1]}\right\} \\
    C^{[t]} = C^{[t-1]}\cup C^{(t)},~1\leq t\leq T.
\end{align*}
Further, define $\phi(C)$ by the collection of forces $u\rightarrow v$ such that $y_{a}^{t}=1$, for some $t\in[T]$ and $a=(u,v)\in A$.
Note that every vertex of $G$ is in exactly one set $C^{(t)}$, where $0\leq t\leq T$.
Also, by constraint~\eqref{eq:tsm_const1}, every $v\in V$ is either in $C^{[0]}$ or corresponds to exactly one time step $t\in[T]$ and exactly one arc $a=(u,v)$ such that $y^{t}_{a}=1$. 
If $y^{t}_{a}=1$, then constraints~\eqref{eq:tsm_const2}--\eqref{eq:tsm_const3} imply that
\[
x_{u}^{t-1}=1
~\text{and}~
x_{w}^{t-1}=1,
\]
for all $w\in N(u)\setminus\{v\}$.
Moreover, constraint~\eqref{eq:tsm_const4} implies that
\[
x_{v}^{t-1}=0~\text{and}~x_{v}^{t}=1.
\]
Hence, $v\in C^{(t)}$ and $u$ and all its neighbors except for $v$ are in $C^{[t-1]}$. 

Now, let $t\in[T]$ and $(u,v)\in A$. 
Suppose that $u$ and all of its neighbors except $v$ are in $C^{[t-1]}$.
Then, by constraint~\eqref{eq:tsm_const5}, there exists an $a=(w,v)\in A$ such that $y_{a}^{t}=1$.
Finally, note that constraints~\eqref{eq:tsm_const6}--\eqref{eq:tsm_const7} imply that $z^{t}=1$ if and only if there exists a $v\in C^{(t)}$.

Therefore, every vertex of $G$ is either initially filled in $C$ or is forced later in $C^{(t)}$, for exactly one $t\in[T]$. 
Furthermore, if $v\in C^{(t)}$, then $v$ must be forced by exactly one if its neighbors $u$ such that $u$ and all of its neighbors except for $v$ are in $C^{[t-1]}$, and the forcing $u\rightarrow v$ is stored in $\phi(C)$.
Conversely, if $u$ and all of its neighbors except for $v$ are in $C^{[t-1]}$, then $v$ must be forced by one of its neighbors $w$, and the forcing $w\rightarrow v$ is stored in $\phi(C)$.
Finally, $z^{t}=1$ if and only if there exists a $v\in C^{(t)}$.
It follows that $C$ is a zero forcing set whose closure occurs in $\sum_{t\in[T]}z^{t}$ time steps, where on each time step all possible forces that can be done independently of each other are applied, and $\phi(C)$ is a collection of forces that produce $\cl{C}$. 
\end{proof}

Next, we show that every zero forcing game in $\mathcal{Z}(G,T)$ corresponds to a feasible solution of $\tsm{G}{T}$.
\begin{theorem}\label{thm:tsm_feas2}
Let $G\in\mathbb{G}$.
For each zero forcing set $C$ corresponding to a zero forcing game in $\mathcal{Z}(G,T)$, there is a feasible solution $(x,y,z)\in\feas{\tsm{G}{T}}$ such that $\abs{C}=\sum_{v\in V}x^{0}_{v}$ and $\sum_{t\in[T]}z^{t}=\pt(G,C)\leq T$.
\end{theorem}
\begin{proof}
Consider the collection of subsets $C^{(t)}$ and $C^{[t]}$, where $0\leq t\leq T$, and the collection of forces $\phi(C)$ that represent the given zero forcing game.
For each $v\in V$, there is exactly one $t\in[T]\cup\{0\}$ such that $v\in C^{(t)}$. 
Define $x_{v}^{t+i}=1$, for all $0\leq i\leq T_t$, and set $x_{v}^{i}=0$, for all $0\leq i<t$. 
If $t=0$, then $v$ was initially filled and therefore never forced.
In this case, define $y_{a}^{i}=0$ for all $a=(u,v)\in A$ and for all $i\in[T]$.
Otherwise, $t\in[T]$ and there is a forcing $u\rightarrow v$ in $\phi(C)$ for exactly one neighbor $u\in N(v)$. 
In this case, define $y_{a}^{t}=1$, where $a=(u,v)$. 
Moreover, for all $i\in[T]\setminus\{t\}$, set $y_{a}^{i}=0$.
Also, for all $w\in N(v)\setminus\{u\}$ and for all $i\in[T]$, set $y_{a}^{i}=0$, where $a=(w,v)$.
Finally, let $t^{*}\in[T]$ such that $C^{[t^{*}]}=V$.
Then, define $z^{i}=1$, for all $1\leq i\leq t^{*}$, and set $z^{i}=0$, for all $t^{*}<i\leq T$. 

For every $v\in V$, $v\in C$ or there is exactly one neighbor $u\in N(v)$ and exactly one time step $t\in[T]$ such that $u\rightarrow v$  is in $\phi(C)$ and $v\in C^{(t)}$. 
Hence, constraint~\eqref{eq:tsm_const1} holds. 
If $u\rightarrow v$ at time step $t\in[T]$, then $u$ and all of its neighbors except for $v$ are in $C^{[t-1]}$; thus, constraints~\eqref{eq:tsm_const2}--\eqref{eq:tsm_const3} hold.
Once $v$ is filled it remains filled and is never again forced, so constraint~\eqref{eq:tsm_const4} holds. 
On each time step, all possible forces that can be done done independently of each other are applied. 
Therefore, for each $t\in[T]$ and each arc $(u,v)\in A$, if $u$ and all of its neighbors except for $v$ are in $C^{[t-1]}$, then $v$ must be in $C^{(t)}$; hence, constraint~\eqref{eq:tsm_const5} holds.
Finally, for each $t\in[T]$, $z^{t}=1$ if and only if there is a $v\in C^{(t)}$, so constraints~\eqref{eq:tsm_const6}--\eqref{eq:tsm_const7} hold. 
\end{proof}
\subsection{Optimal solutions}\label{subsec:time_step_optimal}
We say that $(x,y,z)\in\feas{\tsm{G}{T}}$ is an optimal solution if it minimizes the objective function in~\eqref{eq:tsm_obj}.
We denote by $\opt{\tsm{G}{T}}$ the set of all optimal solutions of $\tsm{G}{T}$.
The following result shows that the optimal solutions of $\tsm{G}{T}$ correspond to minimum zero forcing sets of $G$. 
\begin{corollary}\label{cor:tsm_opt}
Let $G\in\mathbb{G}$ and $T=(n-1)$.
If $(x,y,z)\in\opt{\tsm{G}{T}}$, then $C=\left\{v\in V\colon x^{0}_{v}=1\right\}$ is a minimum zero forcing set of $G$. 
\end{corollary}
\begin{proof}
Let $(x,y,z)\in\opt{\tsm{G}{T}}$.
Then, $(x,y,z)\in\feas{\tsm{G}{T}}$ and is minimizes the objective function in~\eqref{eq:tsm_obj}.
Define $C=\left\{v\in V\colon x^{0}_{v}=1\right\}$.
Since $(x,y,z)\in\feas{\tsm{G}{T}}$, Theorem~\ref{thm:tsm_feas1} implies that $C$ is a zero forcing set corresponding to a zero forcing game in $\mathcal{Z}(G,T)$. 
Therefore, $\Z(G)\leq\sum_{v\in V}x^{0}_{v}$.
For the sake of contradiction, suppose that $\Z(G)<\abs{C}$.
Then, there is a zero forcing set $\hat{C}$ corresponding to a zero forcing game in $\mathcal{Z}(G,T)$ such that
\[
\abs{\hat{C}}=\Z(G) < \abs{C}. 
\]
However, Theorem~\ref{thm:tsm_feas2} implies that there exists a feasible solution $(\hat{x},\hat{y},\hat{z})\in\feas{\tsm{G}{T}}$ such that $\abs{\hat{C}}=\sum_{v\in V}\hat{x}^{0}_{v}$, which contradicts the optimality of $(x,y,z)\in\opt{\tsm{G}{T}}$. 
\end{proof}

In what follows, we make use of three other objective functions to compute the minimum propagation time, maximum propagation time, and throttling number of a graph. 
Specifically, let $\tsmpt{G}{T}$ denote the model in~\eqref{eq:tsm_obj}--\eqref{eq:tsm_const10}, where~\eqref{eq:tsm_obj} is replaced by the objective function
\begin{equation}\label{eq:tsmpt_obj}
\textrm{minimize}~\sum_{v\in V}x^{0}_{v} + \frac{1}{2T}\sum_{t\in[T]}z^{t}.
\end{equation}
Also, let $\tsmPT{G}{T}$ denote the model in~\eqref{eq:tsm_obj}--\eqref{eq:tsm_const10}, where~\eqref{eq:tsm_obj} is replaced by the objective function
\begin{equation}\label{eq:tsmPT_obj}
\textrm{minimize}~\sum_{v\in V}x^{0}_{v}-\frac{1}{2T}\sum_{t\in[T]}z^{t}.
\end{equation}
Finally, let $\tsmth{G}{T}$ denote the model in~\eqref{eq:tsm_obj}--\eqref{eq:tsm_const10}, where~\eqref{eq:tsm_obj} is replaced by the objective function
\begin{equation}\label{eq:tsmth_obj}
\textrm{minimize}~\sum_{v\in V}x^{0}_{v} + \sum_{t\in[T]}z^{t}.
\end{equation}
Note that $\tsm{G}{T}$, $\tsmpt{G}{T}$, $\tsmPT{G}{T}$, and $\tsmth{G}{T}$ all have the same feasible solutions but potentially different optimal solutions. 
The following result shows that the optimal solutions of $\tsmpt{G}{T}$ correspond to minimum zero forcing sets with minimum propagation time. 
\begin{corollary}\label{cor:tsmpt_opt}
Let $G\in\mathbb{G}$ and $T=(n-1)$.
If $(x,y,z)\in\opt{\tsmpt{G}{T}}$, then $C=\left\{v\in V\colon x^{0}_{v}=1\right\}$ is a minimum zero forcing set of $G$ such that $\pt(G)=\pt(G,C)=\sum_{t\in[T]}z^{t}$. 
\end{corollary}
\begin{proof}
Let $(x,y,z)\in\opt{\tsmpt{G}{T}}$.
Then, $(x,y,z)\in\feas{\tsm{G}{T}}$ and is minimizes the objective function in~\eqref{eq:tsmpt_obj}.
Define $C=\left\{v\in V\colon x^{0}_{v}=1\right\}$. 
Since $0\leq\frac{1}{2T}\sum_{t\in[T]}z^{t}\leq\frac{1}{2}$, it follows that the objective function first minimizes $\sum_{v\in V}x^{0}_{v}$ and then minimizes $\sum_{t\in[T]}z^{t}$. 
Hence, by Corollary~\ref{cor:tsm_opt}, $C$ is a minimum zero forcing set of $G$. 
Furthermore, since $(x,y,z)$ is a feasible solution, Theorem~\ref{thm:tsm_feas1} implies that
\[
\pt(G)\leq\pt(G,C)=\sum_{t\in[T]}z^{t}.
\]
For the sake of contradiction, suppose that $\pt(G)<\sum_{t\in[T]}z^{t}$.
Th, there is a zero forcing set $\hat{C}$ corresponding to a zero forcing game in $\mathcal{Z}(G,T)$ such that
\[
\pt(G,\hat{C}) = \pt(G) < \sum_{t\in[T]}z^{t}
~\text{and}~
\abs{\hat{C}}=\Z(G).
\]
However, Theorem~\ref{thm:tsm_feas2} implies that there is a feasible solution $(\hat{x},\hat{y},\hat{z})\in\feas{\tsm{G}{T}}$ such that $\abs{\hat{C}}=\sum_{v\in V}\hat{x}^{0}_{v}$ and $\pt(G,\hat{C})=\sum_{t\in[T]}\hat{z}^{t}$, which contradicts the optimality of $(x,y,z)\in\opt{\tsmpt{G}{T}}$. 
\end{proof}
The following result shows that the optimal solutions of $\tsmPT{G}{T}$ correspond to minimum zero forcing sets with maximum propagation time. 
\begin{corollary}\label{cor:tsmPT_opt}
Let $G\in\mathbb{G}$ and $T=(n-1)$.
If $(x,y,z)\in\opt{\tsmPT{G}{T}}$, then $C=\left\{v\in V\colon x^{0}_{v}=1\right\}$ is a minimum zero forcing set of $G$ such that $\PT(G)=\pt(G,C)=\sum_{t\in[T]}z^{t}$. 
\end{corollary}
\begin{proof}
Let $(x,y,z)\in\opt{\tsmPT{G}{T}}$.
Then, $(x,y,z)\in\feas{\tsm{G}{T}}$ and is minimizes the objective function in~\eqref{eq:tsmPT_obj}.
Define $C=\left\{v\in V\colon x^{0}_{v}=1\right\}$. 
Since $0\leq\frac{1}{2T}\sum_{t\in[T]}z^{t}\leq\frac{1}{2}$, it follows that the objective function first minimizes $\sum_{v\in V}x^{0}_{v}$ and then maximizes $\sum_{t\in[T]}z^{t}$. 
Hence, by Corollary~\ref{cor:tsm_opt}, $C$ is a minimum zero forcing set of $G$. 
Furthermore, since $(x,y,z)$ is a feasible solution, Theorem~\ref{thm:tsm_feas1} implies that
\[
\PT(G)\geq\pt(G,C)=\sum_{t\in[T]}z^{t}.
\]
For the sake of contradiction, suppose that $\PT(G)>\sum_{t\in[T]}z^{t}$.
Then, there exists a zero forcing set $\hat{C}$ corresponding to a zero forcing game in $\mathcal{Z}(G,T)$ such that
\[
\pt(G,\hat{C}) = \PT(G) > \sum_{t\in[T]}\hat{z}^{t}
~\text{and}~
\abs{\hat{C}}=\Z(G).
\]
However, Theorem~\ref{thm:tsm_feas2} implies that there is a feasible solution $(\hat{x},\hat{y},\hat{z})\in\feas{\tsm{G}{T}}$ such that $\abs{\hat{C}}=\sum_{v\in V}\hat{x}^{0}_{v}$ and $\pt(G,\hat{C})=\sum_{t\in[T]}\hat{z}^{t}$, which contradicts the optimality of $(x,y,z)\in\opt{\tsmPT{G}{T}}$. 
\end{proof}
The following result shows that the optimal solutions of $\tsmth{G}{T}$ correspond to zero forcing sets with minimum throttling number. 
\begin{corollary}\label{cor:tsmth_opt}
Let $G\in\mathbb{G}$ and $T=(n-1)$.
If $(x,y,z)\in\opt{\tsmth{G}{T}}$, then $C=\left\{v\in V\colon x^{0}_{v}=1\right\}$ is a zero forcing set of $G$ such that $\th(G)=\th(G,C)=\sum_{v\in V}x^{0}_{v}+\sum_{t\in[T]}z^{t}$.
\end{corollary}
\begin{proof}
Let $(x,y,z)\in\opt{\tsmth{G}{T}}$.
Then, $(x,y,z)\in\feas{\tsm{G}{T}}$ and is minimizes the objective function in~\eqref{eq:tsmth_obj}.
Define $C=\left\{v\in V\colon x^{0}_{v}=1\right\}$. 
Since $(x,y,z)$ is a feasible solution, Theorem~\ref{thm:tsm_feas1} implies that $C$ is a zero forcing set corresponding to a zero forcing game in $\mathcal{Z}(G,T)$, where $\pt(G,C)=\sum_{t\in[T]}z^{t}$. 
Therefore,
\[
\th(G)\leq\th(G,C) = \sum_{v\in V}x^{0}_{v} + \sum_{t\in[T]}z^{t}.
\]
For the sake of contradiction, suppose that $\th(G) < \sum_{v\in V}x^{0}_{v}+\sum_{t\in[T]}z^{t}$.
Then, there exists a zero forcing set $\hat{C}$ corresponding to a zero forcing game in $\mathcal{Z}(G,T)$ such that
\[
\th(G,\hat{C}) = \th(G) < \sum_{v\in V}x^{0}_{v} + \sum_{t\in[T]}z^{t}.
\]
However, Theorem~\ref{thm:tsm_feas2} implies that there exists a feasible solution $(\hat{x},\hat{y},\hat{z})\in\feas{\tsm{G}{T}}$ such that $\hat{C}=\sum_{v\in V}x^{0}_{v}$ and $\pt(G,\hat{C})=\sum_{t\in[T]}\hat{z}^{t}$, which contradicts the optimality of $(x,y,z)\in\opt{\tsmth{G}{T}}$. 
\end{proof}
\subsection{Propagation time interval}\label{subsec:pt_interval}
Let $G\in\mathbb{G}$.
In this section, we introduce a model for computing the realized propagation time interval. 

Given $k\in\pti(G)$, we let $\tsmpti{G}{T}{k}$ denote the model in~\eqref{eq:tsm_obj}--\eqref{eq:tsm_const10}, where~\eqref{eq:tsm_obj} is replaced by the objective function~\eqref{eq:tsmpt_obj} and the following constraint is added:
\begin{equation}\label{eq:tsmpti_const}
\sum_{t\in[T]}z^{t} \geq k. 
\end{equation}

Note that every feasible solution of $\tsmpti{G}{T}{k}$ is a feasible solution of $\tsm{G}{T}$ that also satisfies constraint~\eqref{eq:tsmpti_const}. 
We denote by $\feas{\tsmpti{G}{T}{k}}$ the set of all feasible solutions of $\tsmpti{G}{T}{k}$.
We say that $(x,y,z)\in\feas{\tsmpti{G}{T}{k}}$ is an optimal solution if it minimizes the objective function in~\eqref{eq:tsmpt_obj}.
We denote by $\opt{\tsmpti{G}{T}{k}}$ the set of all optimal solutions of $\tsmpti{G}{T}{k}$.
\begin{corollary}\label{cor:tsmpti_opt}
Let $G\in\mathbb{G}$, $T=(n-1)$, and $k\leq\PT(G)$.
If $(x,y,z)\in\opt{\tsmpti{G}{T}{k}}$, then $C = \left\{v\in V\colon x^{0}_{v}=1\right\}$ is a minimum zero forcing set of $G$ such that $\pt(G,C) = \sum_{t\in[T]}z^{t}\geq k$.
Furthermore, there is no minimum zero forcing set $\hat{C}$ such that $\pt(G,\hat{C})\in\left(k,\pt(G,C)\right)$. 
\end{corollary}
\begin{proof}
Let $(x,y,z)\in\opt{\tsmpti{G}{T}{k}}$.
Then, $(x,y,z)\in\feas{\tsm{G}{T}}$ and is minimizes the objective function in~\eqref{eq:tsmpt_obj}.
Since $0\leq\frac{1}{2T}\sum_{t\in[T]}z^{t}\leq\frac{1}{2}$, it follows that the objective function first minimizes $\sum_{v\in V}x^{0}_{v}$ and then minimizes $\sum_{t\in[T]}z^{t}$. 
Therefore, by Corollary~\ref{cor:tsm_opt}, $C$ is a minimum zero forcing set of $G$. 
Furthermore, since $(x,y,z)$ is a feasible solution, Theorem~\ref{thm:tsm_feas1} implies that $\pt(G,C)=\sum_{t\in[T]}z^{t}$. 
In addition, since $z$ must satisfy constraint~\eqref{eq:tsmpti_const}, it follows that $\pt(G,C)\geq k$. 
For the sake of contradiction, suppose that there is a minimum zero forcing set $\hat{C}$ such that $\pt(G,\hat{C})\in\left(k,\pt(G,C)\right)$.
Then, Theorem~\ref{thm:tsm_feas2} implies that there is a feasible solution $(\hat{x},\hat{y},\hat{z})\in\feas{\tsm{G}{T}}$ such that $\abs{\hat{C}}=\sum_{v\in V}\hat{x}^{0}_{v}$ and $\pt(G,\hat{C})=\sum_{t\in[T]}\hat{z}^{t}$. 
Since $\pt(G,\hat{C})>k$, it follows that constraint~\eqref{eq:tsmpti_const} holds; hence, $(\hat{x},\hat{y},\hat{z})\in\feas{\tsmpti{G}{T}{k}}$, which contradicts the optimality of $(x,y,z)\in\opt{\tsmpti{G}{T}{k}}$. 
\end{proof}

Corollary~\ref{cor:tsmpti_opt} shows that each optimal solution of $\tsmpti{G}{T}{k}$ corresponds to a minimum zero forcing set $C$ such that $\pt(G,C)$ is the nearest realizable propagation time to $k$, not smaller than $k$, over all minimum zero forcing sets of $G$.
Furthermore, if $k>\PT(G)$ then $\tsmpti{G}{T}{k}$ is infeasible.
Therefore, we are able to compute the realized propagation time interval for any graph $G$, as shown in Algorithm~\ref{alg:pti}.
\begin{algorithm}[ht!]
\caption{Realized propagation time interval for graph $G\in\mathbb{G}$ }
\label{alg:pti}
\begin{algorithmic}
\Function{$\pti$}{$G$}
    \State{$T\gets (n-1)$}
    \State{$k\gets 0$}
    \State{$\mathcal{I}\gets\emptyset$}
    \While{$\tsmpti{G}{T}{k}$ is feasible}
        \State{Let $(x,y,z)\in\opt{\tsmpti{G}{T}{k}}$}
        \State{$\mathcal{I}\gets\mathcal{I}\cup\left\{\sum_{t\in[T]}z^{t}\right\}$}
        \State{$k\gets \sum_{t\in[T]}z^{t}+1$}
    \EndWhile
    \State{\Return{$\mathcal{I}$}}
\EndFunction
\end{algorithmic}
\end{algorithm}

\section{Fort Cover Model}\label{sec:fort_cover}
The Fort Cover Model was first proposed in~\cite{Brimkov2019zf} for computing the zero forcing number of a graph using the fort cover definition of zero forcing sets, see Theorem~\ref{thm:fort_cover}.
In this section, we review the Fort Cover Model and show how its linear programming relaxation can be used to compute the fractional zero forcing number, which was first introduced in~\cite{Cameron2023}. 

The Fort Cover Model for a graph $G$ is denoted by $\fc{G}$ and stated in~\eqref{eq:fc_obj}--\eqref{eq:fc_const2}.
Note that the binary variable $s_{v}$ indicates whether the vertex $v$ is in the zero forcing set.
\begin{mini!}
	{}{\sum_{v\in V}s_{v}}{}{}\label{eq:fc_obj}
	\addConstraint{\sum_{v\in F}s_{v}}{\geq 1,~\quad\forall F\in \mathcal{F}_{G}}\label{eq:fc_const1}
    \addConstraint{s}{\in\{0,1\}^{n}}\label{eq:fc_const2}
 \end{mini!}

We say that $s$ is a feasible solution of $\fc{G}$ if constraints~\eqref{eq:fc_const1}--\eqref{eq:fc_const2} are satisfied.
Furthermore, we denote by $\feas{\fc{G}}$ the set of all feasible solutions of $\fc{G}$.
In addition, we say that $s\in\feas{\fc{G}}$ is an optimal solution if it minimizes the objective function in~\eqref{eq:fc_obj}.
We denote by $\opt{\fc{G}}$ the set of all optimal solutions of $\fc{G}$. 
By Theorem~\ref{thm:fort_cover} and constraint~\eqref{eq:fc_const1}, the feasible solutions of $\fc{G}$ correspond to the zero forcing sets of $G$. 
Therefore, $s\in\opt{\fc{G}}$ if and only if $s$ corresponds to a minimum zero forcing set of $G$.
For reference, we state this observation in Corollary~\ref{cor:fort_cover}.
\begin{corollary}\label{cor:fort_cover}
Let $G\in\mathbb{G}$.
If $s\in\opt{\fc{G}}$, then $C=\left\{v\in V\colon s_{v}=1\right\}$ is a minimum zero forcing set of $G$.
\end{corollary}

\subsection{Constraint generation}\label{subsec:fort_cover_cg}
Note that constraint~\eqref{eq:fc_const1} need only consider minimal forts of $G$ since every fort contains a minimal fort. 
However, even this restriction is not enough to ensure a polynomial, in the order of $G$, number of constraints. 
Indeed, there exists families of graphs with an exponential number of minimal forts, such as the path, cycle, wheel, and windmill graphs~\cite{Brimkov2021,Becker2025}.
For this reason, $\fc{G}$ must use constraint generation, see~\cite{Dantzig1954,Nemhauser1999}.
In particular, a relaxed model is obtained from the full model $\fc{G}$ by omitting the constraints in~\eqref{eq:fc_const1}; then, the relaxed model is solved and a set of violated constraints from the full model are added to the relaxed model, and this process is repeated until there are no more violated constraints. 

In~\cite{Brimkov2019zf}, the authors explore methods for generating violated constraints from the full model $\fc{G}$. 
Here, we review one such method. 
Let $s\in\{0,1\}^{n}$ be a solution to the relaxed model, and define $C\subseteq V$ where $v\in C$ if $s_{v}=1$. 
If $C$ is a zero forcing set, no constraints from $\fc{G}$ violate $s$, implying $s\in\opt{\fc{G}}$ and $\Z(G)=\sum_{v\in V}s_{v}$. 
If $C$ is not a zero forcing set, then there exists at least one fort $F\in\mathcal{F}_{G}$ with $F\subset V\setminus{C}$, which we refer to as a \emph{violated fort}. 
The minimum fort model, denoted by $\mf{G}{C}$, computes a violated fort of minimum cardinality, with the binary variable $x_{v}$ indicating vertex inclusion in the fort.

\begin{mini!}
    {}{\sum_{v\in V}x_{v}}{}{}\label{eq:mf_obj}
    \addConstraint{\sum_{v\in V}x_{v}}{\geq 1}\label{eq:mf_const1}
    \addConstraint{x_{u}-x_{v}+\sum_{w\in N(u)\setminus\{v\}}x_{w}}{\geq 0,~\quad\forall v\in V,~u\in N(v)}\label{eq:mf_const2}
    \addConstraint{x_{v}}{=0,~\quad\forall v\in C}\label{eq:mf_const3}
    \addConstraint{x}{\in\{0,1\}^{n}}\label{eq:mf_const4}
\end{mini!}

We say that $x$ is a feasible solution of $\mf{G}{C}$ if constraints~\eqref{eq:mf_const1}--\eqref{eq:mf_const4} are satisfied.
Furthermore, we denote by $\feas{\mf{G}{C}}$ the set of all feasible solutions of $\mf{G}{C}$.
In addition, we say that $x\in\feas{\mf{G}{C}}$ is an optimal solution if it minimizes the objective function in~\eqref{eq:mf_obj}.
We denote by $\opt{\mf{G}{C}}$ the set of all optimal solutions of $\mf{G}{C}$.

Note that $\mf{G}{C}$ is nearly equivalent to Model 3 in~\cite{Brimkov2019zf}, with the exception that constraint~\eqref{eq:mf_const3} has been relaxed from the constraint $x_{v}=0,~\forall v\in\cl{C}$.
Theorem~\ref{thm:mf_feasible} shows that the feasible solutions of $\mf{G}{C}$ are equivalent to the feasible solutions of Model 3 in~\cite{Brimkov2019zf}, which correspond to the collection of forts in $V\setminus\cl{C}$.
We prefer the formulation of $\mf{G}{C}$ since it avoids the computation of the closure of $C$, which while polynomial time can still be costly.
\begin{lemma}\label{lem:mf_feasible}
Let $G\in\mathbb{G}$.
Also, let $x\in\{0,1\}^{n}$ and define $F\subseteq V$ such that $v\in F$ if and only if $x_{v}=1$.
Then, $F$ is a fort of $G$ if and only if constraints~\eqref{eq:mf_const1}--\eqref{eq:mf_const2} hold. 
\end{lemma}
\begin{proof}
Suppose that $F$ is a fort of $G$.
By definition, $F$ is non-empty; hence, constraint~\eqref{eq:mf_const1} holds.
For the sake of contradiction, suppose that constraint~\eqref{eq:mf_const2} does not hold.
Then, there exists a vertex $v\in F$ with a neighbor $u\in N(v)$ such that $u\notin F$ and $w\notin F$ for all neighbors $w\in N(u)\setminus\{v\}$.
However, this implies that there exists a vertex $u\in V\setminus{F}$ that has exactly one neighbor in $F$, which contradicts $F$ being a fort of $G$.

Conversely, suppose that $x$ satisfies constraints~\eqref{eq:mf_const1}--\eqref{eq:mf_const2}.
Then, $F$ is non-empty.
For the sake of contradiction, suppose that there exists a vertex $u\in V\setminus{F}$ such that $N(u)\cap F=\{v\}$.
Then, there is a $v\in F$ such that $v$ has a neighbor $u$ where $u$ and all of its neighbors except for $v$ are not in $F$, which contradicts constraint~\eqref{eq:mf_const2}.
\end{proof}
\begin{theorem}\label{thm:mf_feasible}
Let $G\in\mathbb{G}$ and let $C\subseteq V$.
Also, let $x\in\{0,1\}^{n}$ and define $F\subseteq V$ such that $v\in F$ if and only if $x_{v}=1$. 
Then, $F$ is a fort of $G$ and $F\subseteq V\setminus\cl{C}$ if and only if constraints~\eqref{eq:mf_const1}--\eqref{eq:mf_const3} hold.
\end{theorem}
\begin{proof}
Suppose $F$ is a fort of $G$ and $F\subseteq V\setminus\cl{C}$.
Then, by Lemma~\ref{lem:mf_feasible}, constraints~\eqref{eq:mf_const1}--\eqref{eq:mf_const2} hold.
Moreover, we have
\[
F\subseteq V\setminus\cl{C} \subseteq V\setminus{C},
\]
so constraint~\eqref{eq:mf_const3} also holds.

Conversely, suppose that $x$ satisfies constraints~\eqref{eq:mf_const1}--\eqref{eq:mf_const3}.
Then, $F\subseteq V\setminus{C}$ and, by Lemma~\ref{lem:mf_feasible}, $F$ is a fort of $G$.
For the sake of contradiction, suppose that $F\cap\cl{C}\neq\emptyset$.
Then, there is a vertex not in $F$ that has exactly one neighbor in $F$, which contradicts $F$ being a fort. 
\end{proof}

The following result shows that the optimal solutions of $\mf{G}{C}$ correspond to violated forts of minimum cardinality. 
\begin{corollary}\label{cor:mf_optimal}
Let $G\in\mathbb{G}$ and let $C\subseteq V$.
Also, let $x\in\opt{\mf{G}{C}}$ and $F\subseteq V$ such that $v\in F$ if and only if $x_{v}=1$.
Then, $F$ is a fort of $G$ of minimum cardinality such that $F\subseteq V\setminus\cl{C}$.
\end{corollary}
\begin{proof}
Since $x\in\opt{\mf{G}{C}}\subseteq\feas{\mf{G}{C}}$, Theorem~\ref{thm:mf_feasible} implies that $F$ is a fort of $G$ and $F\subseteq V\setminus\cl{C}$.
Since $x$ is optimal with respect to the objective in~\eqref{eq:mf_obj}, it follows that there is no fort of $G$ in $V\setminus\cl{C}$ with cardinality less than $\abs{F}$. 
\end{proof}
\subsection{Linear programming relaxation}\label{subsec:fort_cover_relax}
The linear programming relaxation of $\fc{G}$ is the model that arises by removing the integrality constraint in~\eqref{eq:fc_const2}; we denote this model by $\lfc{G}$ and state it in~\eqref{eq:lfc_obj}--\eqref{eq:lfc_const2}.
Note that the continuous variable $s_{v}$ indicates the weight of the vertex $v$. 
\begin{mini!}
	{}{\sum_{v\in V}s_{v}}{}{}\label{eq:lfc_obj}
	\addConstraint{\sum_{v\in F}s_{v}}{\geq 1,~\quad\forall F\in \mathcal{F}_{G}}\label{eq:lfc_const1}
    \addConstraint{s}{\in [0,1]^{n}}\label{eq:lfc_const2}
 \end{mini!}

A solution $s$ is feasible for $\lfc{G}$ if it satisfies constraints \eqref{eq:lfc_const1}--\eqref{eq:lfc_const2}. 
The set of feasible solutions is denoted $\feas{\lfc{G}}$.
A feasible solution $s$ corresponds to a weighting of the vertices of $G$ such that the sum of weights over each fort is at least $1$. 
An optimal solution $s \in \feas{\lfc{G}}$ minimizes the objective function in \eqref{eq:lfc_obj}, and is represented by $\opt{\lfc{G}}$.
Such solutions yield the minimum total weight, known as the fractional zero forcing number $\Z^{*}(G)$, as defined in~\cite{Cameron2023}.

Since constraint~\eqref{eq:lfc_const1} requires all minimal forts of $G$, $\lfc{G}$ must use constraint generation.
In particular, a relaxed model is obtained from the full model $\lfc{G}$ by omitting the constraints in~\eqref{eq:lfc_const1}; then, the relaxed model is solved and a set of violated constraints from the full model are added to the relaxed model, and this process is repeated until there are no more violated constraints. Before we further explain how to apply constraint generation, we briefly motivate the necessity of this methodology. 

It is natural to wonder whether constraint generation is needed for $\lfc{G}$, or whether it suffices to apply constraints generated for the Fort Cover Model $\fc{G}$ to the relaxed Fort Cover Model $\lfc{G}$. Example~\ref{ex:non_sufficient_const} shows this is not the case.
 \begin{example}\label{ex:non_sufficient_const}
Consider the graph $C_{5}$, shown in Figure~\ref{fig:c5}, which has minimal forts $F_{1}=\{1,2,4\}$, $F_{2}=\{2,3,5\}$, $F_{3} = \{1,3,4\}$, $F_{4} = \{2,4,5\}$, $F_{5} = \{1, 3, 5\}$.
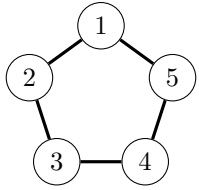
\begin{figure}[ht!]
\centering
\begin{tikzpicture}
[nodeFilledDecorate/.style={shape=circle,inner sep=2pt,draw=black,fill=lightgray,thick},%
    nodeEmptyDecorate/.style={shape=circle,inner sep=2pt,draw=black,thick},%
    lineDecorate/.style={black,=>latex',very thick},%
    x=1cm,y=1cm,scale=1.0]
    \node[circle,draw=black,fill=white] (1) at (0,1) {$1$};
    \node[circle,draw=black,fill=white] (2) at (-0.951,0.309) {$2$};
    \node[circle,draw=black,fill=white] (3) at (-0.588,-0.809) {$3$};
    \node[circle,draw=black,fill=white] (4) at (0.588,-0.809) {$4$};
    \node[circle,draw=black,fill=white] (5) at (0.951,0.309) {$5$};

    \foreach \x/\y in {1/2,2/3,3/4,4/5,5/1}
        \draw[black,=>latex',-,very thick] (\x) -- (\y);
\end{tikzpicture}
\caption{Graph $C_{5}$}
\label{fig:c5}
\end{figure}

Suppose that $\sum_{v\in F_{i}}s_{v}\geq 1$, for $i\in\{1,2,3\}$, are the only constraints given to the Fort Cover Model $\fc{G}$.
Then, $s=\{1,1,0,0,0\}$ is a feasible solution that does not violate any other fort constraint, i.e., $\sum_{v\in F_{i}}s_{v}\geq 1$, for $i\in\{4,5\}$.
However, if $\sum_{v\in F_{i}}s_{v}\geq 1$, for $i\in\{1,2,3\}$, are the only constraints given to the relaxed Fort Cover Model $\lfc{G}$, then $s=\{0.5,0.5,0.5,0,0\}$ is a feasible solution that violates the constraint $\sum_{v\in F_{4}}s_{v}\geq 1$.
\end{example}

Now, we provide details on our constraint generation method for $\lfc{G}$. 
Let $s \in [0,1]^{n}$ be a solution to the relaxed model.
To generate a violated fort, that is, a fort that violates~\eqref{eq:lfc_const1}, we use a modification of the minimum fort model.
We reference this model as the fractional minimum fort model, denote by $\lmf{G}{s}$, and state in~\eqref{eq:lmf_obj}--\eqref{eq:lmf_const4}.
Note that the binary variable $x_{v}$ indicates whether the vertex $v$ is in the fort.

\begin{mini!}
    {}{\sum_{v\in V}s_{v}x_{v}}{}{}\label{eq:lmf_obj}
    \addConstraint{\sum_{v\in V}x_{v}}{\geq 1}\label{eq:lmf_const1}
    \addConstraint{x_{u}-x_{v}+\sum_{w\in N(u)\setminus\{v\}}x_{w}}{\geq 0,~\quad\forall v\in V,~u\in N(v)}\label{eq:lmf_const2}
    \addConstraint{x}{\in\{0,1\}^{n}}\label{eq:lmf_const4}
\end{mini!}

We say that $x$ is a feasible solution of $\lmf{G}{s}$ if constraints~\eqref{eq:lmf_const1}--\eqref{eq:lmf_const4} are satisfied.
Furthermore, we denote by $\feas{\lmf{G}{s}}$ the set of all feasible solutions of $\lmf{G}{s}$.
In addition, we say that $x\in\feas{\lmf{G}{s}}$ is an optimal solution if it minimizes the objective function in~\eqref{eq:lmf_obj}.
We denote by $\opt{\lmf{G}{s}}$ the set of all optimal solutions of $\lmf{G}{s}$.
The following result shows that, under suitable conditions, the optimal solutions of $\lmf{G}{s}$ correspond to violated forts of $\lfc{G}$. 
\begin{theorem}\label{thm:lmf_optimal}
Let $G\in\mathbb{G}$ and let $s\in[0,1]^{n}$.
Then, there is a violated fort if and only if the objective value of all $x\in\opt{\lmf{G}{s}}$ is less than one. 
\end{theorem}
\begin{proof}
Suppose there is a fort $F\subseteq V$ of $G$ such that constraint~\eqref{eq:lfc_const1} does not hold. 
Let $x\in\{0,1\}^{n}$ such that $x_{v}=1$ if and only if $v\in F$. 
By Lemma~\ref{lem:mf_feasible}, $x\in\feas{\lmf{G}{s}}$. 
Furthermore, since constraint~\eqref{eq:lfc_const1} does not hold, it follows that
\[
\sum_{v\in V}s_{v}x_{v} < 1.
\]
So, every $x\in\opt{\lmf{G}{s}}$ has an objective value less than $1$. 

Conversely, suppose that every $x\in\opt{\lmf{G}{s}}$ has an objective value less than one. 
Let $x\in\opt{\lmf{G}{s}}$ and $F\subseteq V$ such that $v\in F$ if and only if $x_{v}=1$.
By Lemma~\ref{lem:mf_feasible}, $F$ is a fort of $G$.
Thus, 
\[
\sum_{v\in F}s_{v} = \sum_{v\in V}s_{v}x_{v} < 1,
\]
so constraint~\eqref{eq:lfc_const1} does not hold. 
\end{proof}

\section{Minimal fort model}\label{sec:min_fort_model}
In this section, we modify the minimum fort model in~\eqref{eq:mf_obj}--\eqref{eq:mf_const4} to obtain a model for the computation of all minimal forts of a graph. 
Given a graph $G\in\mathbb{G}$, a fort $F\subseteq V$ is \emph{minimal} if no fort of $G$ is a proper subset of $F$. 

Let $G\in\mathbb{G}$ and let $\mathcal{F}$ denote a, possibly empty, collection of minimal forts of $G$. 
The minimal fort model with respect to $\mathcal{F}$, which we denote by $\mff{G}{\mathcal{F}}$ and state in~\eqref{eq:mff_obj}--\eqref{eq:mff_const4}, computes a fort of $G$ of minimum cardinality that is not a superset of any fort in $\mathcal{F}$.
Note that the binary variable $x_{v}$ indicates whether the vertex $v$ is in the fort.

\begin{mini!}
    {}{\sum_{v\in V}x_{v}}{}{}\label{eq:mff_obj}
    \addConstraint{\sum_{v\in V}x_{v}}{\geq 1}\label{eq:mff_const1}
    \addConstraint{x_{u}-x_{v}+\sum_{w\in N(u)\setminus\{v\}}x_{w}}{\geq 0,~\quad\forall v\in V,~u\in N(v)}\label{eq:mff_const2}
    \addConstraint{\sum_{v\in F}x_{v}}{\leq\abs{F}-1,~\quad\forall F\in \mathcal{F}}\label{eq:mff_const3} 
    \addConstraint{x}{\in\{0,1\}^{n}}\label{eq:mff_const4}
\end{mini!}

A solution $x$ is feasible for $\mff{G}{\mathcal{F}}$ if it satisfies constraints \eqref{eq:mff_const1}--\eqref{eq:mff_const4}. 
We denote the set of all feasible solutions by $\feas{\mff{G}{\mathcal{F}}}$. 
The following result shows that feasible solutions correspond to forts $F'$ that are minimal with respect to $\mathcal{F}$, meaning $F \not\subseteq F'$ for all $F \in \mathcal{F}$.
\begin{lemma}\label{lem:mff_feasible}
Let $G\in\mathbb{G}$ and $\mathcal{F}$ denote a, possibly empty, collection of forts of $G$.
Also, let $x\in\{0,1\}^{n}$ and $F'\subseteq V$ such that $v\in F'$ if and only if $x_{v}=1$.
Then, $F'$ is a fort of $G$ that is minimal with respect to $\mathcal{F}$ if and only if $x$ satisfies constraints~\eqref{eq:mff_const1}--\eqref{eq:mff_const3}.
\end{lemma}
\begin{proof}
Suppose that $x$ satisfies constraints~\eqref{eq:mff_const1}--\eqref{eq:mff_const2}.
Then, by Lemma~\ref{lem:mf_feasible}, $F'$ is a fort of $G$.
Furthermore, if constraint~\eqref{eq:mff_const3} holds, then
\[
\abs{F'\cap F}\leq\abs{F}-1,
\]
for all $F\in\mathcal{F}$.
Hence, $F\not\subseteq F'$ for all $F\in\mathcal{F}$.
Therefore, $F'$ is a fort of $G$ that is minimal with respect to $\mathcal{F}$.

Conversely, suppose that $F'$ is a fort of $G$ that is minimal with respect to $\mathcal{F}$.
Since $F'$ is a fort of $G$, Lemma~\ref{lem:mf_feasible} implies that $x$ satisfies constraints~\eqref{eq:mff_const1}--\eqref{eq:mff_const2}. 
Furthermore, since $F'$ is minimal with respect to $\mathcal{F}$, it follows that $x$ satisfies constraint~\eqref{eq:mff_const3}.
\end{proof}

A feasible solution $x$ is optimal if it minimizes the objective function in \eqref{eq:mff_obj}; we denote the set of all optimal solutions by $\opt{\mff{G}{\mathcal{F}}}$. 
If $\mathcal{F}$ is a collection of minimal forts of $G$, then the following theorem shows that the optimal solutions can be used to extend $\mathcal{F}$ to a larger collection of minimal forts of $G$.
\begin{theorem}\label{thm:mff_optimal}
Let $G\in\mathbb{G}$ and $\mathcal{F}$ denote a, possibly empty, collection of minimal forts of $G$. 
Suppose that $\mff{G}{\mathcal{F}}$ is feasible and let $x\in\opt{\mff{G}{\mathcal{F}}}$ and $F'\subseteq V$ such that $v\in F'$ if and only if $x_{v}=1$. 
Then, $\mathcal{F}\cup\{F'\}$ is a collection of minimal forts of $G$.
\end{theorem}
\begin{proof}
Since $x\in\opt{\mff{G}{\mathcal{F}}}\subseteq\feas{\mff{G}{\mathcal{F}}}$, Lemma~\ref{lem:mff_feasible} implies that $F'$ is a fort of $G$ that is minimal with respect to $\mathcal{F}$, that is, $F\not\subseteq F'$ for all $F\in\mathcal{F}$. 
Since $\mathcal{F}$ is a collection of minimal forts of $G$, there is no fort of $G$ that is a proper subset of any fort in $\mathcal{F}$.

For the sake of contradiction, suppose there exists a fort $\hat{F}$ of $G$ such that $\hat{F}\subset F'$.
Since $F'$ is minimal with respect to $\mathcal{F}$, it follows that $\hat{F}$ is a fort of $G$ that is minimal with respect to $\mathcal{F}$. 
Now, define $\hat{x}\in\{0,1\}^{n}$ by $\hat{x}_{v}=1$ if and only if $v\in \hat{F}$. 
Then, by Lemma~\ref{lem:mff_feasible}, $\hat{x}\in\feas{\mff{G}{\mathcal{F}}}$ and 
\[
\sum_{v\in V}\hat{x}_{v} < \sum_{v\in V}x_{v},
\]
which contradicts the optimality of $x$ with respect to the objective function in~\eqref{eq:mff_obj}.
\end{proof}

Using Theorem~\ref{thm:mff_optimal}, we are able to recursively compute all minimal forts of any graph $G$, as shown in Algorithm~\ref{alg:amf}. 
\begin{algorithm}[H]
\caption{All minimal forts of a graph $G\in\mathbb{G}$}
\label{alg:amf}
\begin{algorithmic}
\Function{$\amf$}{$G$}
    \State{$\mathcal{F}\gets\emptyset$}
    \While{$\mff{G}{\mathcal{F}}$ is feasible}
        \State{Let $x\in\opt{\mff{G}{\mathcal{F}}}$}
        \State{$\mathcal{F}\gets\mathcal{F}\cup\left\{v\in V\colon x_{v}=1\right\}$}
    \EndWhile
    \State{\Return{$\mathcal{F}$}}
\EndFunction
\end{algorithmic}
\end{algorithm}
\begin{corollary}\label{cor:all_min_forts}
Let $G\in\mathbb{G}$ and $\mathcal{F}=\amf(G)$.
Then, $\mathcal{F}$ is the collection of all minimal forts of $G$. 
\end{corollary}
\begin{proof}
Since $\mathcal{F}$ starts as the empty set in Algorithm~\ref{alg:amf}, Theorem~\ref{thm:mff_optimal} implies that every fort in $\mathcal{F}$ is a minimal fort of $G$. 
For the sake of contradiction, suppose there exists a minimal fort $F'\subseteq V$ of $G$ that is not in $\mathcal{F}$. 
Then, define $x\in\{0,1\}^{n}$ by $x_{v}=1$ if and only if $v\in F'$. 
Since $F'$ is a minimal fort of $G$, it follows that $F'$ is minimal with respect to $\mathcal{F}$.
Therefore, Lemma~\ref{lem:mff_feasible} implies that $x\in\feas{\mff{G}{\mathcal{F}}}$, which contradicts the completion of Algorithm~\ref{alg:amf}.
\end{proof}
\section{Fort Number Model}\label{sec:fort_number}
Note that the fort number of a graph can be computed via the classical set packing problem~\cite{Nemhauser1999}, given the collection of all minimal forts of $G$ which can be computed via Algorithm~\ref{alg:amf}.
However, there are known families of graphs where the number of minimal forts grows exponentially with the order of the graph, for example, the path, cycle, and windmill graphs~\cite{Becker2025,Brimkov2021}.
In this section, we present a method for computing the fort number of a graph that does not require all minimal forts.
We refer to this method as the Fort Number Model, which we denote by $\fn{G}$ and state in~\eqref{eq:fn_obj}--\eqref{eq:fn_const5}. 
Note that the binary variable $x_{iv}$ indicates whether vertex $v$ is in the $i$th set in the collection; also, the binary variable $z_{i}$ indicates whether the $i$th set is non-empty.

\begin{maxi!}
	{}{\sum_{i=1}^{n}z_{i}}{}{}\label{eq:fn_obj}
    \addConstraint{z_{i}-\sum_{u\in V}x_{iu}}{\leq 0,~\quad\forall i\in\{1,\ldots,n\}}\label{eq:fn_const1}
	\addConstraint{x_{iu}-x_{iv}+\sum_{w\in N(u)\setminus\{v\}}x_{iw}}{\geq 0,~\quad\forall i\in\{1,\ldots,n\},~\forall v\in V,~\forall u\in N(v)}\label{eq:fn_const2}
    \addConstraint{\sum_{i=1}^{n}x_{iu}}{\leq 1,~\quad\forall u\in V}\label{eq:fn_const3}
    \addConstraint{x}{\in \{0,1\}^{n\times n}}\label{eq:fn_const4}
    \addConstraint{z}{\in \{0,1\}^{n}}\label{eq:fn_const5}
\end{maxi!}
\subsection{Feasible solutions}
We say that $(x,z)$ is a feasible solution of $\fn{G}$ if constraints~\eqref{eq:fn_const1}--\eqref{eq:fn_const5} are satisfied.
Furthermore, we denote by $\feas{\fn{G}}$ the set of all feasible solutions of $\fn{G}$.
The following result shows that every feasible solution of $\fn{G}$ corresponds to a collection of pairwise disjoint forts.
\begin{theorem}\label{thm:fn_feas1}
Let $G\in\mathbb{G}$.
For each feasible solution $(x,z)\in\feas{\fn{G}}$ there is a collection of disjoint forts $\mathcal{F}$ such that $\abs{\mathcal{F}}= \sum_{i=1}^{n}z_{i}$.
\end{theorem}
\begin{proof}
Let $(x,z)\in\feas{\fn{G}}$.
Then, for each $i\in\{1,\ldots,n\}$, define $F_{i}\subseteq V$ such that $u\in F_{i}$ if and only if $x_{iu}=1$.
By constraint~\eqref{eq:fn_const1}, if $z_{i}=1$ then there exists a $u\in V$ such that $x_{iu}=1$.
Therefore, by constraint~\eqref{eq:fn_const2}, if $z_{i}=1$ then Lemma~\ref{lem:mf_feasible} implies that $F_{i}$ is a fort of $G$. 
Moreover, by constraint~\eqref{eq:fn_const3}, the subsets $F_{i}$ are pairwise disjoint.
Now, let $\mathcal{F}$ denote the collection of subsets $F_{i}$ such that $z_{i}=1$.
Then, $\mathcal{F}$ is a collection of pairwise disjoint forts such that $\abs{\mathcal{F}}=\sum_{i=1}^{n}z_{i}$. 
\end{proof}

Next, we show that every collection of disjoint forts corresponds to a feasible solution of $\fn{G}$. 
\begin{theorem}\label{thm:fn_feas2}
Let $G\in\mathbb{G}$.
For each collection of pairwise disjoint forts $\mathcal{F}$, there is a feasible solution $(x,z)\in\feas{\fn{G}}$ such that $\abs{\mathcal{F}}=\sum_{i=1}^{n}z_{i}$. 
\end{theorem}
\begin{proof}
Let $\mathcal{F}=\{F_{1},\ldots,F_{k}\}$ denote a collection of disjoint forts.
Then, define $x\in\{0,1\}^{n\times n}$ and $z_{i}\in\{0,1\}^{n}$ as follows:
For $i=1,\ldots,k$, set $z_{i}=1$ and define $x_{iv}=1$ if and only if $v\in F_{i}$.
For $i=k+1,\ldots,n$, set $z_{i}=0$ and $x_{iv}=0$ for all $v\in V$. 

Since there is a $u\in V$ such that $x_{iu}=1$ whenever $z_{i}=1$, it follows that constraint~\eqref{eq:fn_const1} holds.
Furthermore, if $z_{i}=1$ then $F_{i}$ is a fort; hence, Lemma~\ref{lem:mf_feasible} implies that constraint~\eqref{eq:fn_const2} holds. 
Also, since $\mathcal{F}$ is a collection of pairwise disjoint forts and $x_{iu}=0$ for all $i=k+1,\ldots,n$ and $u\in V$, it follows that constraint~\eqref{eq:fn_const3} holds. 
Therefore, $(x,z)\in\feas{\fn{G}}$ and $\abs{\mathcal{F}}=\sum_{i=1}^{n}z_{i}$.
\end{proof}
\subsection{Optimal solutions}
We say that $(x,z)\in\feas{\fn{G}}$ is an optimal solution if it maximizes the objective function in~\eqref{eq:fn_obj}.
We denote by $\opt{\fn{G}}$ the set of all optimal solutions of $\fn{G}$.
The following result shows that the optimal solutions of $\fn{G}$ correspond to maximum collections of pairwise disjoint forts. 
\begin{corollary}\label{cor:fn_opt}
Let $(x,z)\in\opt{\fn{G}}$. 
Then, $\ft(G) = \sum_{i=1}^{n}z_{i}$.
\end{corollary}
\begin{proof}
Let $(x,z)\in\opt{\fn{G}}$.
Then, $(x,z)\in\feas{\fn{G}}$ and is maximal with respect to the objective function in~\eqref{eq:fn_obj}. 
Since $(x,z)$ is a feasible solution, Theorem~\ref{thm:fn_feas1} implies that there is a collection of disjoint forts $\mathcal{F}$ such that $\abs{\mathcal{F}}=\sum_{i=1}^{n}z_{i}$. 
Therefore, $\ft(G)\geq \sum_{i=1}^{n}z_{i}$.
For the sake of contradiction, suppose that $\ft(G) > \sum_{i=1}^{n}z_{i}$.
Then, there exists a collection of disjoint forts $\hat{\mathcal{F}}$ such that
\[
\abs{\hat{\mathcal{F}}}=\ft(G) > \sum_{i=1}^{n}z_{i}.
\]
However, Theorem~\ref{thm:fn_feas2} implies that there exists a feasible solution $(\hat{x},\hat{z})$ such that $\abs{\hat{\mathcal{F}}} = \sum_{i=1}^{n}\hat{z}_{i}$, which contradicts the optimality of $(x,z)\in\opt{\fn{G}}$.
\end{proof}
\section{Numerical Experiments}\label{sec:num_exp}
In this section, we provide several numerical experiments that demonstrate the effectiveness of our models when working with small and medium order graphs.
Moreover, we provide experimental evidence for several open conjectures regarding the propagation time interval, the number of minimal forts, the fort number, and the fractional zero forcing number of a graph~\cite{Becker2025,Brimkov2025,Cameron2023}.
Our code is written in C++ and uses the Gurobi Optimizer and is available at~\url{https://github.com/trcameron/ZFIPModels}. 
Our computational results were obtained on a MacBook Pro with an Apple M3 Pro chip and 18 gigabytes of RAM, while the code was compiled using Apple clang version 15.0.0 and Gurobi version 11.0.3.

For each graph tested, we use the Infection Model from Section~\ref{sec:infection} to compute the zero forcing number, minimum propagation time, and throttling number; we use the Time Step Model from Section~\ref{sec:time_step} to compute the zero forcing number, minimum and maximum propagation times, and throttling number; we use the Fort Cover Model and its relaxation from Section~\ref{sec:fort_cover} to compute the zero forcing number and fractional zero forcing number, respectively.
Also, we compute the propagation time interval using Algorithm~\ref{alg:pti}, all minimal forts of a graph using Algorithm~\ref{alg:amf}, and the fort number of a graph using the Fort Number Model in Section~\ref{sec:fort_number}. 

The Gurobi Time Limit parameter is set to $7200$ seconds.
In the case that this time limit is reached, the reported time will be `Time Limit'.
\subsection{Small order graphs}
In this section, we demonstrate the efficiency of our models when working with small order graphs.
In particular, we use the Nauty Traces software from~\cite{Mckay2014} to generate all non-isomorphic graphs of order $4\leq n\leq 9$.
For $n=9$, we use edge bounds to parallelize the generated graphs and the related computations.
In Table~\ref{tab:zf_small_order}, we compare the average time elapsed to compute the zero forcing number and fractional zero forcing number via the Infection Model, Time Step Model, and Fort Cover Model and its relaxation.
\begin{table}[ht!]
    \centering
    \resizebox{0.75\textwidth}{!}{%
    \begin{tabular}{c|c|c|c|c}
    $n$ &  Infection Model & Time Step Model & Fort Cover Model & Fort Cover Model Relaxation \\
    \hline
    4 & 0.001 & 0.001 & 0.002 & 0.001 \\
    5 & 0.002 & 0.002 & 0.002 & 0.002 \\
    6 & 0.003 & 0.005 & 0.003 & 0.002 \\
    7 & 0.007 & 0.016 & 0.006 & 0.005 \\
    8 & 0.014 & 0.057 & 0.011 & 0.011 \\
    9 & 0.035 & 0.211 & 0.029 & 0.028 \\
    \hline
    \end{tabular}%
    }
    \caption{Average time elapsed for computing the zero forcing number and fractional zero forcing number of all non-isomorphic graphs of order $4\leq n\leq 9$.}
    \label{tab:zf_small_order}
\end{table}

In Table~\ref{tab:pt_PT_th_small_order}, we compare the average time elapsed to compute the minimum and maximum propagation times and throttling number of a graph via the Infection Model and Time Step Model. 
\begin{table}[ht!]
    \centering
    \resizebox{0.95\textwidth}{!}{%
    \begin{tabular}{c|c|c|c|c|c}
    $n$ & Infection Model (pt) & Infection Model (th) & Time Step Model (pt) & Time Step Model (PT) & Time Step Model (th) \\
    \hline
    4 & 0.001 & 0.001 & 0.001 & 0.001 & 0.001 \\
    5 & 0.002 & 0.002 & 0.002 & 0.002 & 0.003 \\
    6 & 0.005 & 0.004 & 0.006 & 0.006 & 0.008 \\
    7 & 0.012 & 0.010 & 0.023 & 0.019 & 0.031 \\
    8 & 0.027 & 0.023 & 0.081 & 0.068 & 0.110 \\
    9 & 0.066 & 0.059 & 0.269 & 0.219 & 0.367 \\
    \hline
    \end{tabular}%
    }
    \caption{Average time elapsed for computing the minimum and maximum propagation times and throttling number of all non-isomorphic graphs of order $4\leq n\leq 9$.}
    \label{tab:pt_PT_th_small_order}
\end{table}

In Table~\ref{tab:PTI_fzf_amf_ft_small_order}, we compare the average time elapsed to compute the propagation time interval, all minimal forts, and the fort number of a graph via Algorithm~\ref{alg:pti}, Algorithm~\ref{alg:amf}, and the Fort Number Model, respectively. 
\begin{table}[ht!]
    \centering
    \resizebox{0.45\textwidth}{!}{%
    \begin{tabular}{c|c|c|c}
    $n$ & Algorithm~\ref{alg:pti} & Algorithm~\ref{alg:amf} & Fort Number Model \\
    \hline 
    4 & 0.001 & 0.001 & 0.001 \\
    5 & 0.003 & 0.001 & 0.003 \\
    6 & 0.010 & 0.002 & 0.007 \\
    7 & 0.044 & 0.009 & 0.023 \\
    8 & 0.181 & 0.025 & 0.049 \\
    9 & 0.711 & 0.103 & 0.117 \\
    \end{tabular}%
    }
    \caption{Average time elapsed for computing the realized propagation time interval, all minimal forts, and the fort number of all non-isomorphic graphs of order $4\leq n\leq 9$.}
    \label{tab:PTI_fzf_amf_ft_small_order}
\end{table}
\subsection{Medium order random graphs}
In this section, we provide several numerical experiments that demonstrate the efficiency of our models when working with medium order graphs.
In particular, we use the Nauty Traces software from~\cite{Mckay2014} to generate $10$ random graphs of order $n\in\{10,15\}$ and edge probability $p\in\{0.2,0.3,0.4,0.5,0.6,0.7,0.8\}$.
In Table~\ref{tab:zf_random}, we compare the average time elapsed to compute the zero forcing number of a graph via the Infection Model, Time Step Model, and Fort Cover Model.
\begin{table}[ht!]
    \centering
    \resizebox{0.80\textwidth}{!}{%
    \begin{tabular}{c|c|c|c|c|c}
    $n$ &  $p$ & Infection Model & Time Step Model & Fort Cover Model & Fort Cover Model Relaxation \\
    \hline
    & 0.2 & 0.003 & 0.010 & 0.003 & 0.004 \\
    & 0.3 & 0.011 & 0.164 & 0.009 & 0.013 \\
    & 0.4 & 0.055 & 0.455 & 0.054 & 0.052 \\
    10 & 0.5 & 0.066 & 0.686 & 0.068 & 0.077 \\
    & 0.6 & 0.091 & 0.752 & 0.075 & 0.061 \\
    & 0.7 & 0.137 & 0.970 & 0.119 & 0.070 \\
    & 0.8 & 0.140 & 0.781 & 0.101 & 0.058 \\
    \hline
    & 0.2 & 0.391 & 8.280 & 0.571 & 0.847 \\
    & 0.3 & 1.290 & 39.90 & 1.780 & 1.470 \\
    & 0.4 & 3.520 & 60.10 & 2.580 & 2.060 \\
    15 & 0.5 & 4.540 & 78.10 & 4.040 & 2.140 \\
    & 0.6 & 14.00 & 87.60 & 4.720 & 2.140 \\
    & 0.7 & 16.60 & 83.70  & 3.250 & 1.530 \\
    & 0.8 & 23.90 & 81.90 & 3.690 & 1.020 \\
    \end{tabular}%
    }
    \caption{Average time elapsed for computing the zero forcing number and fractional zero forcing number of $10$ random graphs of order $n\in\{10,15\}$ and edge probability $p\in\{0.2,0.3,0.4,0.5,0.6,0.7,0.8\}$.}
    \label{tab:zf_random}
\end{table}

In Table~\ref{tab:pt_PT_th_random}, we compare the average time elapsed to compute the minimum and maximum propagation times and throttling number of a graph via the Infection Model and Time Step Model. 
\begin{table}[ht!]
    \centering
    \resizebox{1.0\textwidth}{!}{%
    \begin{tabular}{c|c|c|c|c|c|c}
    $n$ &  $p$ & Infection Model (pt) & Infection Model (th) & Time Step Model (pt) & Time Step Model (PT) & Time Step Model (th) \\
    \hline
    & 0.2 & 0.004 & 0.005 & 0.015 & 0.011 & 0.022 \\
    & 0.3 & 0.023 & 0.027 & 0.232 & 0.141 & 0.372 \\
    & 0.4 & 0.099 & 0.095 & 0.648 & 0.465 & 1.220 \\
    10 & 0.5 & 0.134 & 0.095 & 0.836 & 0.588 & 1.320 \\
    & 0.6 & 0.155 & 0.146 & 0.942 & 0.693 & 1.700 \\
    & 0.7 & 0.219 & 0.191 & 1.040 & 0.703 & 1.490 \\
    & 0.8 & 0.205 & 0.198 & 0.764 & 0.669 & 1.260 \\
    \hline
    & 0.2 & 0.640 & 0.627 & 11.30 & 7.740 & 21.40 \\
    & 0.3 & 2.560 & 2.200 & 42.60 & 36.60 & 69.40 \\
    & 0.4 & 6.360 & 3.540 & 71.00 & 57.90 & 117.0 \\
    15 & 0.5 & 9.200 & 5.410 & 76.00 & 59.10 & 136.0 \\
    & 0.6 & 16.00 & 9.920 & 82.00 & 64.40 & 141.0 \\
    & 0.7 & 27.30 & 13.80 & 73.30 & 55.90 & 133.0 \\
    & 0.8 & 45.80 & 31.30 & 90.90 & 72.30 & 139.0 \\
    \end{tabular}%
    }
    \caption{Average time elapsed for computing the minimum and maximum propagation times and throttling number of $10$ random graphs of order $n\in\{10,15\}$ and edge probability $p\in\{0.2,0.3,0.4,0.5,0.6,0.7,0.8\}$.}
    \label{tab:pt_PT_th_random}
\end{table}

In Table~\ref{tab:PTI_fzf_amf_ft_random}, we compare the average time elapsed to compute the propagation time interval, fractional zero forcing number, all minimal forts, and the fort number of a graph via Algorithm~\ref{alg:pti}, the Fort Cover Model relaxation, Algorithm~\ref{alg:amf}, and the Fort Number Model, respectively. 
\begin{table}[ht!]
    \centering
    \resizebox{0.50\textwidth}{!}{%
    \begin{tabular}{c|c|c|c|c}
    $n$ & $p$ & Algorithm~\ref{alg:pti} & Algorithm~\ref{alg:amf} & Fort Number Model \\
    \hline
    & 0.2 & 0.025 & 0.003 & 0.026 \\
    & 0.3 & 0.723 & 0.045 & 0.075 \\
    & 0.4 & 2.320 & 0.269 & 0.226 \\
    10 & 0.5 & 3.040 & 0.408 & 0.199 \\
    & 0.6 & 3.180 & 0.378 & 0.310 \\
    & 0.7 & 3.010 & 0.608 & 0.463 \\
    & 0.8 & 2.180 & 0.446 & 0.656 \\
    \hline
    & 0.2 & 56.70 & 6.960 & 2.480 \\
    & 0.3 & 190.0 & 54.20 & 7.350 \\
    & 0.4 & 341.0 & 127.0 & 26.40 \\
    15 & 0.5 & 363.0 & 148.0 & 75.70 \\
    & 0.6 & 298.0 & 141.0 & 855.0 \\
    & 0.7 & 271.0 & 94.10 & 1370 \\
    & 0.8 & 272.0 & 75.7 & Time Limit \\
    \end{tabular}%
    }
    \caption{Average time elapsed for computing the realized propagation time interval, all minimal forts, and the fort number of $10$ random graphs of order $n\in\{10,15\}$ and edge probability $p\in\{0.2,0.3,0.4,0.5,0.6,0.7,0.8\}$.}
    \label{tab:PTI_fzf_amf_ft_random}
\end{table}
\subsection{The Propagation time interval}
In this section, we use Algorithm~\ref{alg:pti} to provide experimental evidence for Conjecture~\ref{con:pti_hypercube}.
In Table~\ref{tab:pti_hypercube}, we display the minimum and maximum propagation times for $Q_{d}$, for $d\in\{2,3,4,5\}$, along with the realized propagation time interval.
By examination, it can be seen that Conjecture~\ref{con:pti_hypercube} holds for dimensions $2\leq d\leq 5$.
\begin{table}[ht!]
    \centering
    \resizebox{0.5\textwidth}{!}{%
    \begin{tabular}{c|c|c|c}
    $d$ & $\pt(Q_{d})$ & $\PT(Q_{d})$ & Realized Propagation Time Interval \\
    \hline 
    $2$ & $1$ & $1$ & $\{1\}$ \\
    $3$ & $1$ & $2$ & $\{1,2\}$ \\
    $4$ & $1$ & $4$ & $\{1,2,3,4\}$ \\
    $5$ & $1$ & $8$ & $\{1,2,3,4,5,6,7,8\}$ \\
    \end{tabular}%
    }
    \caption{The minimum and maximum propagation times for $Q_{d}$, for $d\in\{2,3,4,5\}$, along with the realized propagation time interval.}
    \label{tab:pti_hypercube}
\end{table}

It is worth noting that we needed to reduce the parameter $T$ from its default value of $n-1$ in order to compute the realized propagation time interval, when $d=5$, without Gurobi reaching its set time limit.
In particular, we set $T=2^{d-1}$, which is valid since every minimum zero forcing set of the hypercube graph $Q_{d}$ has cardinality $2^{d-1}$, thus leaving $2^{d-1}$ vertices remaining to be forced. 
\subsection{The number of minimal forts}
In this section, we use Algorithm~\ref{alg:amf} to provide experimental evidence for Conjecture~\ref{con:tree_gr}.
Let $T_{n}$ denote a tree with a maximum number of minimal forts over all trees of order $n$.
In Table~\ref{tab:amf_tree}, we compare the number of minimal forts of $T_{n}$ to the number of minimal forts of $P_{n}$.
It turns out that the tree $T_{n}$ is unique, at least for $n\leq 23$. 
For this reason, we include its maximum degree and diameter in Table~\ref{tab:amf_tree}. 

\begin{table}[ht!]
    \centering
    \resizebox{0.40\textwidth}{!}{%
    \begin{tabular}{c|c|c|c|c|c}
    $n$ & $\abs{\mathcal{F}_{P_{n}}}$ & $\abs{\mathcal{F}_{T_{n}}}$ & $\frac{\abs{\mathcal{F}_{T_{n}}}}{\abs{\mathcal{F}_{P_{n}}}}$ & $\Delta(T_{n})$ & $\diam(T_{n})$ \\
    \hline
    $16$ & $49$ & $105$ & $2.143$ & $15$ & $2$ \\
    $17$ & $65$ & $120$ & $1.846$ & $16$ & $2$ \\
    $18$ & $86$ & $136$ & $1.581$ & $17$ & $2$ \\
    $19$ & $114$ & $162$ & $1.421$ & $5$ & $4$ \\
    $20$ & $151$ & $213$ & $1.411$ & $5$ & $4$ \\
    $21$ & $200$ & $280$ & $1.400$ & $5$ & $4$ \\
    $22$ & $265$ & $348$ & $1.313$ & $6$ & $4$ \\
    $23$ & $351$ & $453$ & $1.291$ & $5$ & $4$ \\
    \end{tabular}%
    }
    \caption{The maximum number of minimal forts over all trees of order $16\leq n\leq 23$ is compared to the number of minimal forts of $P_{n}$.}
    \label{tab:amf_tree}
\end{table}

We use the formula in~\cite[Corollary 14]{Becker2025} for the number of minimal forts of $P_{n}$.
For $T_{n}$, we use the Nauty Traces software from~\cite{Mckay2014} to generate all non-isomorphic trees of order $n\geq 1$.
Then, we use Algorithm~\ref{alg:amf} to compute all minimal forts of each tree, and we  record $T_{n}$ as the tree with the maximum number of minimal forts. 
For larger $n$, we use diameter bounds to parallelize the generated trees and the related computations. 

For $3\leq n\leq 18$, $T_{n}$ is the star graph of order $n$ and has $\binom{n-1}{2}$ minimal forts. 
For this reason, we did not include the values in Table~\ref{tab:amf_tree} for $n<16$.
For $19\leq n\leq 22$, $T_{n}$ is a tree with a single central vertex adjacent to $4$ junction vertices, each of which is adjacent to $3\leq l\leq 5$ leaves. 

With regards to Conjecture~\ref{con:tree_gr}, it is worth noting that the ratio $\frac{\abs{\mathcal{F}_{T_{n}}}}{\abs{\mathcal{F}_{P_{n}}}}$ is clearly bounded above by $\binom{n}{2}$, for $1\leq n\leq 23$.
Moreover, the ratio $\frac{\abs{\mathcal{F}_{T_{n}}}}{\abs{\mathcal{F}_{P_{n}}}}$ appears to be approaching $1$, which would imply that the limit of the ratio of consecutive terms for $\abs{\mathcal{F}_{T_{n}}}$ is equal to the limit of the ratio of consecutive terms for $\abs{\mathcal{F}_{P_{n}}}$, that is, the plastic ratio $\psi$.
\subsection{Lower bounds on the maximum nullity}
In this section, we provide experimental evidence for Conjecture~\ref{con:mn_lower_bounds}.
Note that this conjecture holds for all graphs of order seven or less since for these graphs the maximum nullity is equal to the zero forcing number. 
In~\cite{Barrett2025}, the authors determine the maximum nullity of all graphs of order eight.
In particular, see Theorem~\ref{thm:max_nullity8}, they identify a collection of graphs $\mathcal{E}=\left\{E_{1},E_{2},\ldots,E_{8}\right\}$ such that for all graphs $G$ of order eight, $\M(G)<\Z(G)$ if and only if $G\in\mathcal{E}$.
Moreover, they show that $\M(E_{i})=\Z(E_{i})-1$, for all $1\leq i\leq 7$.
In Table~\ref{tab:bhhs_max_nullity}, we display the computed values of the fort number and fractional zero forcing number along with the known values of the maximum nullity and zero forcing number of $E_{1},E_{2},\ldots,E_{7}$.
Since $\Z^{*}(E_{i}) < M(E_{i})$, for all $1\leq i\leq 7$, it follows that Conjecture~\ref{con:mn_lower_bounds} holds for all graphs of order eight or less. 
\begin{table}[ht!]
    \centering
    \resizebox{0.30\textwidth}{!}{%
    \begin{tabular}{c|c|c|c|c}
    $E_{i}$ & $\ft(E_{i})$ & $\Z^{*}(E_{i})$ & $\M(E_{i})$ & $\Z(E_{i})$ \\
    \hline 
    $E_{1}$ & $1$ & $1.666667$ & $2$ & $3$ \\
    $E_{2}$ & $2$ & $2.666667$ & $3$ & $4$ \\
    $E_{3}$ & $2$ & $2.666667$ & $3$ & $4$ \\
    $E_{4}$ & $2$ & $2.666667$ & $3$ & $4$ \\
    $E_{5}$ & $2$ & $2.666667$ & $3$ & $4$ \\
    $E_{6}$ & $2$ & $2.666667$ & $3$ & $4$ \\
    $E_{7}$ & $2$ & $2.666667$ & $3$ & $4$ \\
    \end{tabular}%
    }
    \caption{The fort number, fractional zero forcing number, maximum nullity, and zero forcing number of the seven graphs $E_{1},E_{2},\ldots,E_{7}$ of order eight identified in~\cite{Barrett2025}.}
    \label{tab:bhhs_max_nullity}
\end{table}

We further demonstrate the validity of Conjecture~\ref{con:mn_lower_bounds} by showing it holds for vertex and edge sums of the graphs $E_{1},E_{2},\ldots,E_{7}$. 
Theorem~\ref{thm:mr_vs} shows how to compute the minimum rank of the direct sum of disjoint graphs.
Similarly, Theorem~\ref{thm:mr_es} shows how to compute the minium rank of the edge sum of disjoint graphs. 
Note that $r_{v}(G)=\mr(G)-\mr(G-v)$ denotes the \emph{rank spread of} of $G$ at $v$. 
\begin{theorem}[{\cite[Theorem 2.3]{Barioli2004}}]\label{thm:mr_vs}
Let $G_{1},\ldots,G_{h}$ be disjoint graphs and for each $i$ select $v_{i}\in V(G_{i})$.
Let $G$ be formed by joining all $G_{i}'s$ and identifying all $v_{i}'s$ as a unique vertex $v$. 
Then, 
\[
\mr(G) = \sum_{i=1}^{h}\mr(G_{i}-v) + \min\left\{\sum_{i=1}^{h}r_{v}(G_{i}),2\right\}.
\]
\end{theorem}
\begin{theorem}[{\cite[Theorem 2.6]{Barioli2004}}]\label{thm:mr_es}
Let $G=(V,E)$ and $G'=(V',E')$ be disjoint graphs and select $u\in V$ and $u'\in V'$. 
Let $G=G+_{e}G'$.
Then,
\[
\mr(G) = \begin{cases} \mr(G) + \mr(G') & \text{if $r_{u}(G)=2$ or $r_{u'}(G')=2$,} \\ \mr(G) + \mr(G') + 1 & \text{otherwise.}\end{cases}
\]
\end{theorem}

For each $1\leq i\leq 7$, $\mr(E_{i})$ is known.
Moreover, since $E_{i}-v$ is an order $7$ graph its minimum rank can be computed from its zero forcing number.
Hence, we are able to apply Theorem~\ref{thm:mr_vs} and Theorem~\ref{thm:mr_es} to compute the minimum rank of the vertex sum and edge sum, respectively, of the graphs $E_{1},E_{2},\ldots,E_{7}$.

In Figure~\ref{fig:bhhs_test}, we display the fort number, fractional zero forcing number, maximum nullity, and zero forcing number of $E_{i}\oplus_{v} E_{j}$ (left) and $E_{i}+_{e}E_{j}$ (right), for all $1\leq i,j\leq 7$, where $v$ and $e=\{u,u'\}$ are formed from randomly selected $u\in V(E_{i})$ and $u'\in V(E_{j})$. 
By examining the vertical grid lines, one can readily determine that Conjecture~\ref{con:mn_lower_bounds} holds for all generated graphs. 
\begin{figure}[ht!]
    \centering 
    \includegraphics[width=0.425\textwidth]{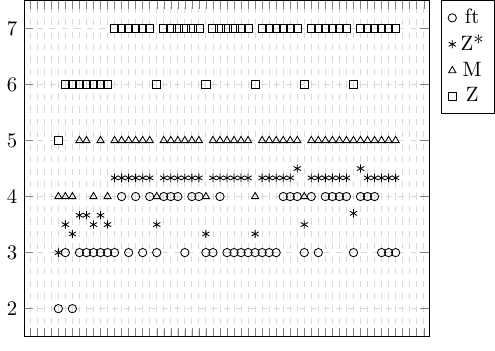}
    \hfill
    \includegraphics[width=0.425\textwidth]{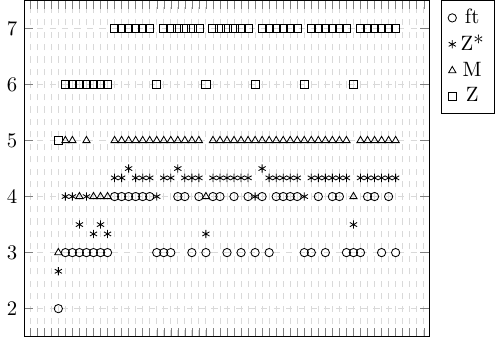}
    \caption{The fort number, fractional zero forcing number, maximum nullity, and zero forcing number of $E_{i}\oplus_{v} E_{j}$ (left) and $E_{i}+_{e}E_{j}$ (right), for all $1\leq i,j\leq 7$, where $v$ and $e=\{u,u'\}$ are formed from randomly selected $u\in V(E_{i})$ and $u'\in V(E_{j})$.}
    \label{fig:bhhs_test}
\end{figure}
\bibliographystyle{siam}
\bibliography{Bibliography}
\end{document}